\documentclass[11pt]{article}
\usepackage[mathcal]{euscript}
\usepackage{bm,url}                     
\usepackage{amsfonts}
\RequirePackage{amsthm,amsmath,amsfonts,amssymb}
\RequirePackage[numbers]{natbib}
\RequirePackage[colorlinks,citecolor=blue,urlcolor=blue]{hyperref}
\RequirePackage{graphicx}

\usepackage{subcaption}
\usepackage{ulem}
\usepackage[dvipsnames]{xcolor}

\newtheorem{theorem}{Theorem}[section]
\newtheorem{lemma}[theorem]{Lemma}
\newtheorem{remark}{Remark}

\newtheorem{corollary}{Corollary}
\theoremstyle{remark}
\newtheorem{definition}[theorem]{Definition}

\newcommand{\eps}{\varepsilon}

\newcommand{\B}{\mathcal{B}}
\newcommand{\R}{\mathbb{R}}

\newcommand{\prob}{\mathbb{P}}
\newcommand{\X}{\mathbb{X}_n}

\usepackage{todonotes}             

\begin{document}
	
	\begin{center}
		\large \bf   Estimation of surface area.\normalsize
	\end{center}
	\normalsize

	\begin{center}
		Catherine Aaron$^a$, Alejandro Cholaquidis$^b$  and Ricardo Fraiman$^b$\\
		$^a$ Universit\'e Clermont Auvergne, France\\
		$^b$ Centro de Matem\'atica, Universidad de la Rep\'ublica, Uruguay\\
	\end{center}

\begin{abstract}
	We study the problem of estimating the surface area of the boundary $\partial S$ of a sufficiently smooth set
$S\subset\mathbb{R}^d$ when the available 
information is only a finite subset  $\X\subset S$.
We propose two estimators. The first makes use of the Devroye--Wise support estimator and is based on Crofton's formula, which, roughly speaking,
states that the $(d-1)$-dimensional surface area of a smooth enough set is the mean  number 
of intersections of randomly chosen lines.  For that purpose, we propose an estimator of the number of intersections of such lines with 
support based on the Devroye--Wise support estimators.  
The second surface area estimator makes use of the $\alpha$-convex hull of $\X$, which is denoted by
$C_{\alpha}(\X)$. More precisely, it is the $(d-1)$-dimensional surface area of $C_\alpha(\X)$, as denoted by $|C_\alpha(\X)|_{d-1}$, which is proven to
converge to the $(d-1)$-dimensional  surface area of $\partial S$.
Moreover, $|C_\alpha(\X)|_{d-1}$ can be computed using Crofton's formula. 

Our results depend on the Hausdorff distance between $S$ and $\X$ for the Devroye--Wise estimator, 
and the Hausdorff distance between $\partial S$ and $\partial C_{\alpha}(\X)$ for the second estimator. 
\end{abstract}

\subsection{On surface area and length estimation}

The estimation of surface areas has been extensively considered in stereology (see, for instance, \cite{bad86,bad05} and \cite{go90}).
It has also been studied as a further step in the theory of nonparametric set estimation (see  \cite{plrc08}),
and has practical applications in medical imaging (see \cite{cfrc:07}).  
In addition,  the estimation of a surface area is widely used in magnetic resonance imagining techniques  (see \cite{guney}).

The three- and two-dimensional  cases are  addressed in  \cite{ber14}, which  proposed parametric estimators
when the available data are the distances to $S$ from a sample outside the set but at a distance smaller than a given $R>0$.

The two-dimensional case has many important applications. This is also true of the three-dimensional case. For instance, surface area  is an important biological parameter in organs such as the lungs (see, for instance, \cite{sar21}). %
The higher dimensional study is also important, at least from a theoretical point of view, because in \cite{Penrose} it is shown that
the boundary surface plays an important role
as a parameter of a probability distribution, which allows us to apply plug-in methods.
To our knowledge, the only paper that tackles the  surface area estimation problem in any dimension, when only ``inside'' data are available, 
is \cite{cuepat18} and no convergence rates are given.

When, as in image analysis, one can observe $n$ data points from  two distinguishable sets of random data-points
(one from inside $S$  and the other from outside $S$), then the estimation
of the surface area of the boundary has been tackled, for any $d\geq 2$, in 
\cite{cfrc:07,cue13,jim11,plrc08} 
and \cite{Yuk16}.
The proposals given in \cite{cfrc:07,plrc08} and \cite{cue13} aim to estimate the Minkowski content 
of $\partial S$. In \cite{cue13},
a very general convergence result is obtained, and in \cite{cfrc:07} a convergence rate of order $n^{-1/2d}$ is obtained under 
some mild hypotheses, and later on, in 
\cite{plrc08}, a convergence rate of order  $(\log(n)/n)^{1/(d+1)}$ is achieved under stronger assumptions.
In \cite{jim11}, a very nice fully data-driven method that is based on the Delaunay triangulation is proposed under an homogeneous point 
process sampling scheme.  The asymptotic rate of convergence of the variance is given but there is no global convergence rate because
no result is obtained for the bias. Finally, in \cite{Yuk16}, a parameter-free procedure that is based on the Voronoi triangulation is proposed
and a rate of convergence  of order $\lambda^{-1/d}$ is obtained under a Poisson Point Process (PPP) sampling scheme (where $\lambda$ 
is the intensity of the PPP).

\subsection{Roadmap}

When $S\subset \mathbb{R}^d$ is a compact set, we aim to estimate its surface area; that is, the $(d-1)$-Hausdorff measure of its boundary $\partial S$.

We propose two surface area estimators,  
at any finite dimension, when the available data is only a finite set $\X\subset S$.
In this setting,  the two-dimensional case has mostly been  studied. 
Assuming that $\X$ is an iid sample, the convex case was 
first addressed in \cite{bh:98} (using Crofton's formula). Later on, under the $\alpha$-convexity assumption, \cite{alphshapr} obtained the convergence of the
$\alpha$-shape's perimeter to the perimeter of the support and the associated convergence rates are derived. When the data  are given by a trajectory  from a reflected Brownian motion (RBM)
(with or without drift), a consistency result is obtained in Theorem 4 in \cite{ch:16}.

Proposed estimator relies on Crofton's formula, which was proven  in 1868 for  convex subsets of $\mathbb{R}^2$ and  extended to arbitrary dimensions
(see \cite{sa:04}).  It states that the surface area of $\partial S$  equals the integral of the number of intersections
with $\partial S$  of lines in $\mathbb{R}^d$ (see Equations \eqref{croft} and \eqref{gencroft} for explicit versions of Crofton's
formula for $d=2$ and $d\geq 2$, respectively).

The first proposed estimator is based on the Devroye--Wise support estimator 
\begin{equation} \label{dw}
	\hat{S}_{\eps_n}(\X)=\cup_{i=1}^n \mathcal{B}(X_i,\eps_n)
\end{equation}
see \cite{DW}, where $n$ is the cardinality of $\X$, $\eps_n\to 0$ as $n\to \infty$ and $\mathcal{B}(X_i,\eps_n)$ denotes the closed ball in $\mathbb{R}^d$ centred at $X_i$ and of radius
$\eps_n>0$.
By use of $\hat{S}_{\eps_n}(\X)$ and $\hat{S}_{4\eps_n}(\X)$ we propose an estimator of the number of intersection of a line with $\partial S$.
The reader should be aware that this estimator is not just a just a plug-in method because (in general) the number of intersections of a line with $\partial \hat{S}_{\eps_n}(\X)$
may not converge to the number of intersections of that line with $\partial S$. The main results regarding this estimator are stated in subsection 
\ref{mainDW} where it is proven that this estimator converges at a rate that is proportional to $d_H(\X,S)^{1/2}$ (where $d_H$ denotes the Hausdorff distance). This rate can be improved to 
$d_H(\X,S)$ when adding a reasonable assumption on the shape of $\partial S$. These rates are known when $\X$ is an iid sample, see Corollary \ref{mainth1bis}. 
The computational aspects of this estimator are studied in subsection \ref{alg}.

The second uses the $\alpha$-convex hull support estimator
\begin{equation} \label{alphahull}
	C_{\alpha}(\X)= \bigcap_{\{x:d(x,\X)\geq\alpha\}} \mathring{\mathcal{B}}(x,\alpha)^c,
\end{equation}
see \cite{rhull}, where $\mathring{\mathcal{B}}(x,\alpha)^c$ denotes the complement of the open ball in $\mathbb{R}^d$ centred at $x$ and of radius $\alpha>0$.  
First we extend the results in \cite{cue12}. 
More precisely, we prove that, in any dimension,  the   surface  area of the hull's boundary---that is, $|\partial C_\alpha(\X)|_{d-1}$---converges to $|\partial S|_{d-1}$.
This result is interesting in itself but in practice it is difficult to compute  $|\partial C_\alpha(\X)|_{d-1}$, especially for dimension $d> 2$.
However, we will see that by means of Crofton's formula it can easily be estimated via the Monte Carlo method. The approach based on the $\alpha$-convex hull
is introduced in Section \ref{alphsec}. A discussion of the rates of convergence is given in Section \ref{rates} and an algorithm based on the Monte Carlo method for the estimator based on the $\alpha$-hull is introduced in Section \ref{algo}.

These results can be applied to many deterministic or random situations to obtain explicit convergence rates.
We focus on two random situations: the case $\X=\{X_1,\ldots,X_n\}$ of iid drawn on $S$ (with a density bounded from below by a positive constant),  
and the case of  random trajectories of reflected diffusions  on $S$.
In particular, we provide  convergence rates when the trajectory is the result of a RBM 
(see  \cite{ch:16,ch:20}).
This last setting has several  applications in  ecology, where the trajectory is obtained by recording the location of an animal (or several animals) living in an area $S$, which is called its home range (the territorial range of the animal),  and $X_t$ represent the position at time $t$ transmitted by the instrument
(see, for instance,  \cite{amparo,ch:16,ch:20},  and the references therein). 

The rate of convergence of the surface area estimator based on $\hat{S}_{\eps_n}(\X)$, when $\X$ is an iid sample, is of order $n^{-1/2d}$. This can be improved to  $n^{-1/d}$, depending on the assumptions  on the smoothness  of $\partial S$. With the estimation of the support  that uses the $\alpha$-convex hull, when $\X$ is an iid sample, we obtain a rate of order $n^{-2/(d+1)}$.

\section{Background} 
\label{back}
\subsection{Notations}

Given a set $S\subset \mathbb{R}^d$, we  denote by
$\mathring{S}$, $\overline{S}$ and $\partial S$ the interior, closure and boundary of $S$,
respectively, with respect to the usual topology of $\mathbb{R}^d$.
We also write $\text{diam}(S)=\sup_{(x,y)\in S\times S} ||x-y||$.
The parallel set of $S$ of radius $\eps$ is 
$B(S,\eps)=\{y\in{\mathbb R}^d:\ \inf_{x\in S}\Vert y-x\Vert\leq \eps \}$.

If $A\subset\mathbb{R}^d$ is a Borel set, then  $|A|_d$ denotes 
its $d$-dimensional Lebesgue measure.
When $A\subset \mathbb{R}^d$ is a $(d-1)$-dimensional manifold, then $|A|_{d-1}$ denotes its 
$(d-1)$-Hausdorff measure.

We  denote by $\mathcal{B}_d(x,\varepsilon)$ (or sometimes just $\mathcal{B}(x,\varepsilon)$) the closed ball
in $\mathbb{R}^d$,
of radius $\varepsilon$, centred at $x$, and  $\omega_d=|\mathcal{B}_d(x,1)|_d$.
Given two compact non-empty sets $A, C \subset{\mathbb R}^d$, 
the Hausdorff distance between $A$ and $C$ is defined by
\begin{equation*}
	d_H(A,C)=\inf\{\eps>0: \mbox{such that } A\subset B(C,\eps)\, \mbox{ and }
	C\subset B(A,\eps)\}.\label{Hausdorff}
\end{equation*}

The  $(d-1)$-dimensional sphere in $\mathbb{R}^d$ is denoted by $\mathcal{S}^{d-1}$, while the half-sphere in $\mathbb{R}^d$ is denoted by $(\mathcal{S}^+)^{d-1}$; that is, 
$(\mathcal{S}^+)^{d-1}=(\mathbb{R}^{d-1}\times \mathbb{R}^+)\cap \mathcal{S}^{d-1}$.
Given $M$ a sufficiently smooth $(d-1)$-manifold and $x\in M$, the affine tangent
space of $M$ at $x$ is denoted by $T_x M$. When $S\subset \mathbb{R}^d$ is regular (i.e., compact and satisfying $S=\overline{\mathring{S}}$) and  has
a $\mathcal{C}^1$ regular boundary $\partial S$, then for any $x\in \partial S$ we can define $\eta_x$ 
the outward normal unit vector at $x$; that is, the unit vector of $(T_x \partial S)^{\perp}$
such that, for $t>0$ small enough, $x+t\eta_x\in S^c$.

Given a vector $\theta\in (\mathcal{S}^+)^{d-1} $ and a point $y$, $r_{\theta,y}$ denotes the line $\{y+\lambda\theta,\lambda\in \mathbb{R}\}=y+\mathbb{R}\theta$.
If $y_1$ and $y_2$ are two points  in $r_{\theta,y}$,    then $y_i=y+\lambda_i\theta$. With a slight abuse of notation, we write $y_1<y_2$ when $\lambda_1<\lambda_2$.

\subsection{Crofton's formula}\label{crof}

In 1868, Crofton proved the following result (see \cite{cro68}): given a convex  set in the plane, whose boundary is denoted by $\gamma$, then its length  
$ |\gamma|_1$ can be computed by 
\begin{equation}\label{croft}
	|\gamma|_1=\frac{1}{2} \int_{\theta=0}^\pi  \int_{p=-\infty}^{+\infty} n_{\gamma}(\theta,p)dpd\theta,
\end{equation}

\noindent $n_{\gamma}(\theta,p)$ being the number of intersections of $\gamma$ with the line 
$r_{\theta^*,\theta p}$,   where $\theta^*\in (\mathcal{S}^+)^{1}$ is orthogonal to $\theta$, and $dpd\theta$ is the two-dimensional Lebesgue measure (see Figure \ref{crofig}).
This result has been generalized to   compact (not necessarily convex) sets in   $\mathbb{R}^d$ for any $d>2$, and also to Lie groups (see \cite{sa:04}).

\begin{figure}[ht]%
	\centering
	\includegraphics[scale=.7]{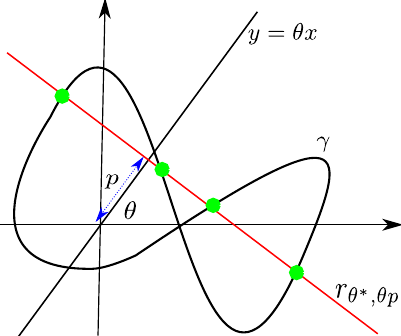} \hspace{3cm} 
	\caption{The function $n_\gamma$ counts the number of intersections of $\gamma$ with the line $r_{\theta^*,\theta p}$ determined by $\theta$ and $p$.}
	\label{crofig}
\end{figure} 

To introduce the general Crofton formula in $\mathbb{R}^d$ for a compact $(d-1)$-dimensional manifold $M$, let us define first the constant 
$$\beta(d)=\Gamma(d/2)\Gamma((d+1)/2)^{-1}\pi^{-1/2},$$ 
where $\Gamma$ stands for the well-known Gamma function.
Let $\theta\in (\mathcal{S}^+)^{d-1}$. Then, $\theta$ determines a $(d-1)$-dimensional linear space $\theta^{\perp}=\{v:\langle v, \theta\rangle=0\}$.
Given $y\in \theta^{\perp}$, let us write $n_{M}(\theta,y)=\#(r_{\theta,y}\cap M)$, where $\#$ is the cardinality of the set (see Figure \ref{cof2}).

\begin{figure}[ht]%
	\centering
	\includegraphics[scale=.5]{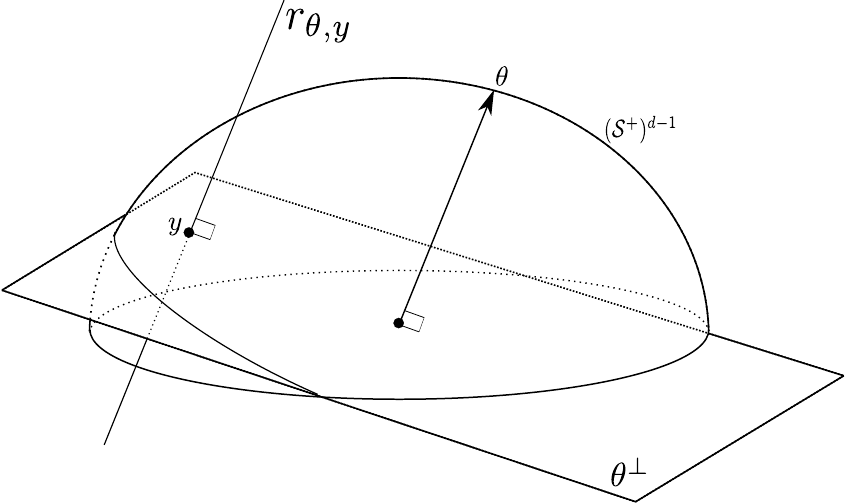} \hspace{3cm} \ 
	\caption{The line $r_{\theta,y}=y+ \mathbb{R}\theta$ is shown, where $y\in \theta^\perp$ and $\theta\in(\mathcal{S}^+)^{d-1}$.}
	\label{cof2}
\end{figure}

It is proven in \cite{fed:69} (see Theorem 3.2.26) that if $M$ is a $(d-1)$-dimensional rectifiable set,
then the integralgeometric measure of $M$ (which will be denoted by $I_{d-1}(M)$,
and is defined by the right-hand side of \eqref{gencroft}) equals its $(d-1)$-dimensional Hausdorff measure; that is,

\begin{equation}\label{gencroft}
	|M|_{d-1}=I_{d-1}(M)=\frac{1}{\beta(d)} \int_{\theta \in {(\mathcal{S}^+)}^{d-1}}\int_{y\in \theta^\perp}  n_{M}(\theta,y)d\mu_{d-1}(y)d\theta.
\end{equation}

The measure $d\theta $ is the uniform measure on $(\mathcal{S}^+)^{d-1}$ (with total mass 1) and $\mu_{d-1}$ is the $(d-1)$-dimensional Lebesgue measure.

\subsection{Restrictions on the shape}

We will now recall some well-known restrictions that are put on the shape in the set estimation.

\begin{definition} \label{def-rconvexity} For $\alpha>0$, a set $S\subset \mathbb{R}^d$ is said to be $\alpha$-convex if 
	$S=C_\alpha(S),$ where $C_\alpha(S)$ is the $\alpha$-convex hull of $S$, defined in \eqref{alphahull}, replacing $\X$ by $S$.
\end{definition}

When $S$ is $\alpha$-convex, a natural estimator of $S$ from a random sample $\X$ of points 
(drawn from a distribution with support $S$), is $C_\alpha(\X)$ (see \cite{rhull}).

\begin{definition} \label{rolling} 
	A set $S\subset \mathbb{R}^d$ is said  to satisfy the outside $\alpha$-rolling condition if for each boundary point $s\in \partial S$ 
	there exists an $x\in S^c$ such that $\mathcal{B}(x,\alpha)\cap \partial S=\{s\}$.
	A compact set $S$ is said to satisfy the inside $\alpha$-rolling condition if $\overline{S^c}$ satisfies the outside 
	$\alpha$-rolling condition.
\end{definition}

Following the notation in \cite{fed:56}, let $\text{Unp}(S)$ be the set of points $x\in \mathbb{R}^d$ with a unique projection on $S$.
\begin{definition} \label{reach} For $x\in S$, let {\it{reach}}$(S,x)=\sup\{r>0:\mathring{\mathcal{B}}(x,r)\subset \text{Unp}(S)\big\}$. The reach of $S$ is defined by $reach(S)=\inf\big\{reach(S,x):x\in S\big\},$ while $S$ is of positive reach if $reach(S)>0$.
\end{definition}

\begin{remark} \label{rem1} Throughout this paper, we  assume that $\partial S$ is the boundary of a compact set
	$S\subset\mathbb{R}^d$ such that $S=\overline{\mathring{S}}$.
	We also assume that $S$ fulfills the outside and inside $\alpha$-rolling conditions, and
	then $\partial S$ is rectifiable (see Theorem 1 in \cite{wal99}).
	From this it follows that   $I_{d-1}(\partial S)=|\partial S|_{d-1}<\infty$, which implies (by \eqref{gencroft}) that, except for a set of measure zero with respect to $d\mu_{d-1}(y) d\theta $, any line $r_{\theta,y}$ meets $\partial S$ a finite 	number of times:  $ n_{\partial S}(\theta,y)<\infty$.
	From Theorem 1 in \cite{wal99}, it also follows  that $\partial S$ is a $\mathcal{C}^1$ manifold, which allows us for each $x\in \partial S$ to define its unit outward normal vector $\eta_x$.
\end{remark}

For the estimator of the surface area based on the Devroye--Wise estimator, we will assume that $\partial S$ satisfies a technical hypothesis, which is referred to as $(C,\eps_0)$-regularity.

\begin{definition}  
	
	Let $E_{\theta}(\partial S)=\{x \in \partial S, \langle \eta_x, \theta \rangle=0\}$. The image of $E_{\theta}(\partial S)$ 
	by the orthogonal projection onto $\theta^\perp$ is denoted by $F_{\theta}=\pi_{\theta^{\perp}}(E_{\theta}(\partial S))$   (which for non-degenerate cases is a $(d-2)$-dimensional submanifold of 
	$\theta^\perp$). We also denote by $B(F_{\theta},\eps)$  its parallel set of radius $\eps$.
	
	We define, for $\eps>0$, 
	\begin{equation*}
		\varphi_{\theta}(\eps)=\big|\theta^\perp\cap B(F_{\theta},\eps)\big|_{d-1}.
	\end{equation*}
	
	\begin{itemize}
		\item We will say that $\partial S$ is $(C,\eps_0)$-regular if for all $\theta$ and all $\eps\in (0,\eps_0)$, 
		$\varphi'_{\theta}(\eps)$ exists and $\varphi'_{\theta}(\eps)\leq C$.
		\item If $\partial S$ is $(C,\eps_0)$-regular for some $\eps_0>0$, then we will say that  $\partial S$ is $C$-regular.
	\end{itemize}
	
\end{definition}

Once the rolling balls condition is imposed, we will show through some examples in Figure  \ref{fig1a} that the $(C,\eps_0)$-regularity of the boundary is quite mild.

\begin{figure}[h]
	\begin{minipage}[b]{0.5\linewidth}
		\centering \includegraphics[scale=0.3]{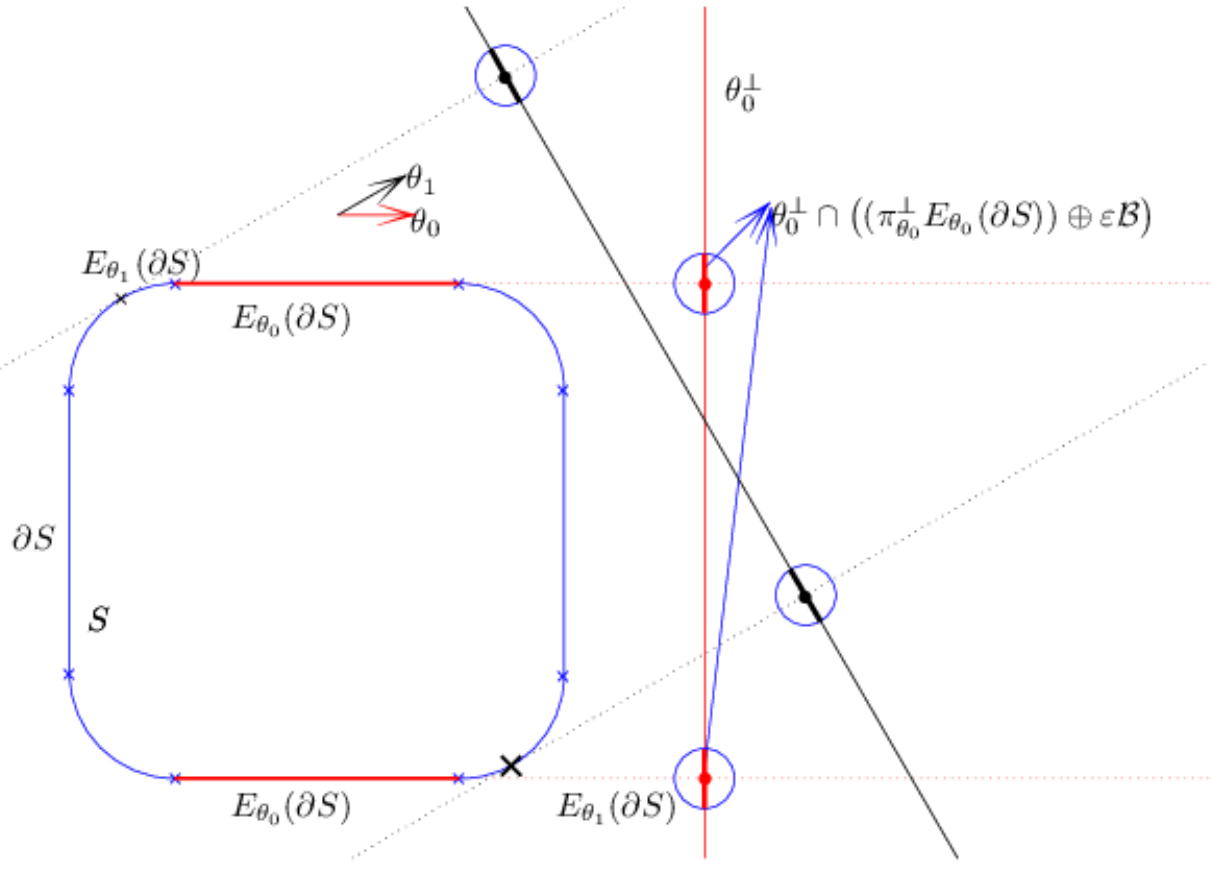}
		\caption{$(a)$ smooth square}
		\label{fig1a}
	\end{minipage}\hfill
	\begin{minipage}[b]{0.5\linewidth}
		\begin{center}
			\includegraphics[scale=0.3]{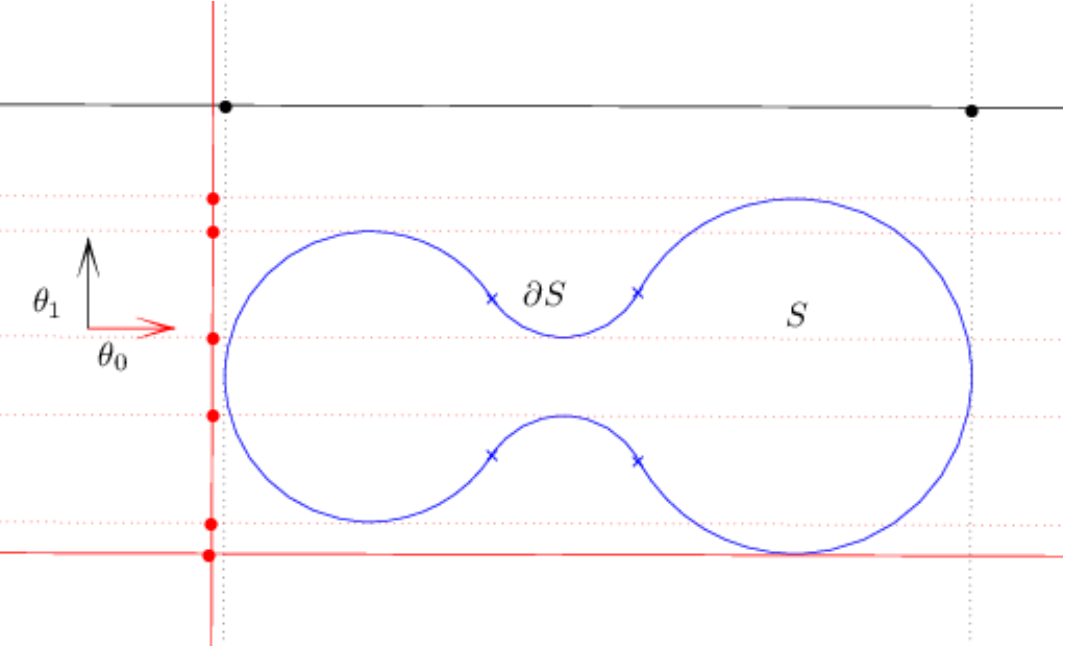}
		\end{center}
		\caption{$(b)$ two-dimensional peanut}
		\label{fig1b}
	\end{minipage}\hfill
	\begin{minipage}[b]{0.5\linewidth}
		\includegraphics[scale=0.25]{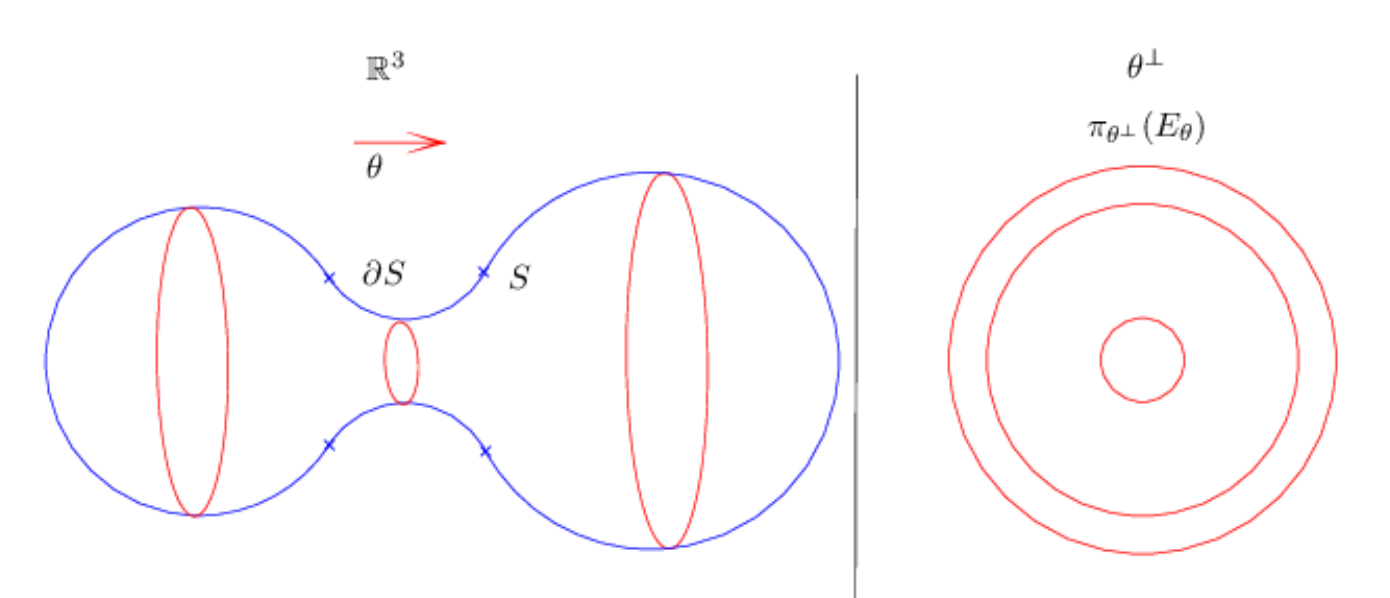}
		\caption{$(c)$ three-dimensional peanut}
		\label{fig1c}
	\end{minipage}\hfill
	\begin{minipage}[b]{0.5\linewidth}
		\begin{center}
			\includegraphics[scale=0.25]{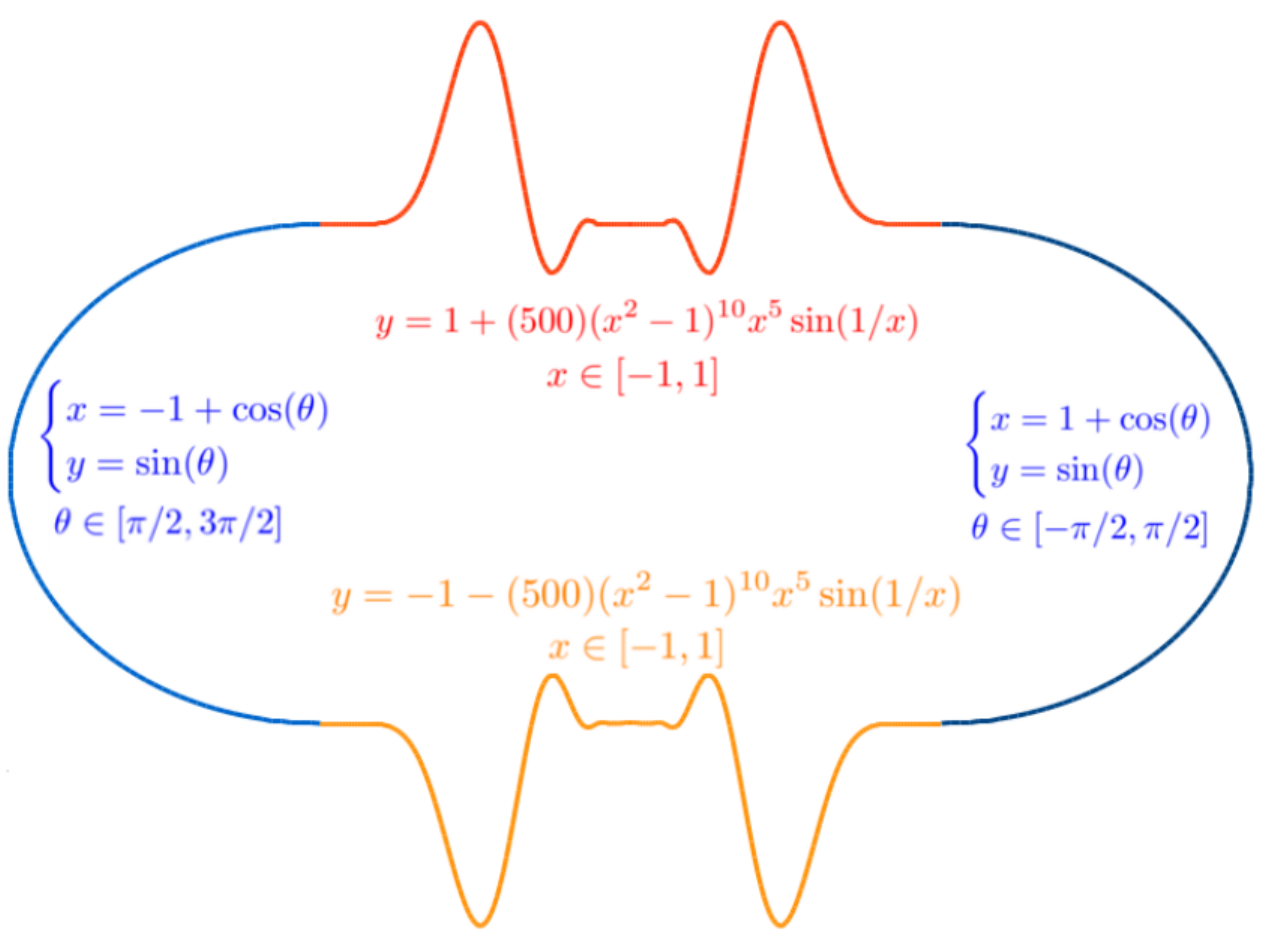}
			\caption{$(d)$ an `infinite wave' shape}  
			\label{fig1d}
		\end{center}
	\end{minipage}\hfill
\end{figure}
\begin{enumerate}
	\item[(a)] The first set, which is presented in Figure \ref{fig1a},   is a unit square with `round angles'. For all $\theta$, 
	$F_{\theta}=\pi_{\theta^{\perp}}(E_{\theta}(\partial S))=\{x_1(\theta),x_2(\theta)\}$ with 
	$||x_1(\theta)-x_2(\theta)||\geq 1$.  Thus, for $\eps<1/2$, and for all 
	$\theta$,  $\varphi_{\theta}(\eps)= 4 \eps$ and is thus $\partial S$ is $(4,0.5)$-regular (in particular $4$-regular).  
	
	\item[(b)] The second set, which is presented in Figure \ref{fig1b}, is a two-dimensional `peanut' that is made of $4$ circular arcs. For all $\theta$ and $\eps$ small enough, we have 
	$\varphi_{\theta}(\eps)= 2 c_{\theta}  \eps$ where $c_{\theta}$ is the number of connected components of $F_{\theta}$, which is less than $6$, from which it   follows   that $S$
	has a $12$-regular boundary.
	\item[(c)] The third set, presented in Figure \ref{fig1c}, is the surface of revolution generated by (b).
	Here we have that for all $\theta$, $E_{\theta}$ is a one-dimensional manifold with less than three connected components.
	The maximal length of a component is bounded by $L$, the length of the maximal perimeter (shown in blue in the figure).
	The reach of each  $E_{\theta}$ is (uniformly in $\theta$) lower bounded by $\alpha>0$.
	All of these assertions allow us to claim that $\partial S$ is $6L$-regular.

	\item[(d)] The rolling ball condition is not sufficient to guarantee the $(C,\eps_0)$-regularity of the boundary:
	this fails if, for instance, $S$ is such that $\partial S=S_1\cup S_2\cup S_3\cup S_4$ (see Figure \ref{fig1d}) with:
	$$S_1=\{(x,1+(500)(x^2-1)^{10}x^5\sin(1/x)), x\in[-1,1]\setminus \{0\}\}\cup\{(0,1)\};$$
	$$S_2=\{(x,-1-(500)(x^2-1)^{10}x^5\sin(1/x)), x\in[-1,1]\setminus \{0\}\}\cup\{(0,-1)\};$$    
	$$S_3=\{(1+\cos(\theta),\sin(\theta)),\theta\in[-\pi/2,\pi/2]\};$$
	$$S_4=\{(-1+\cos(\theta),\sin(\theta)),\theta\in[\pi/2,3\pi/2]\}.$$
	It can easily be proven that such a set satisfies the rolling ball condition for any $r_0\leq 1/80$  
	but $\varphi_0'(\eps)\rightarrow +\infty$ when $\eps\rightarrow 0$, which implies that $\partial S$ is not $C$-regular.
	
\end{enumerate}

For  the Devroye--Wise type estimator, we will also show that the convergence rate can be quadratically improved
if we additionally assume that the number of 
intersections between any line and $\partial S$ is bounded from above (this excludes the case of a linear part in $\partial S$,   such as in Figure \ref{fig1a}).

\begin{definition}  \label{boundinter}
	Given $S\subset \mathbb{R}^d$, we say that $\partial S$ has a bounded number of linear intersections if there exists an $N_S$ such that
	for all $\theta\in (\mathcal{S}^+)^{d-1}$ and $y\in \theta^\perp$, $n_{\partial S}(\theta,y)\leq N_S$.
\end{definition}

\begin{remark}
	The previous definition can be replaced with a weaker requirement by asking that $\partial S$ has a bounded number of linear intersections for almost all lines with respect to $\mu_{d-1}(y)d\theta$, and the corresponding results remain true.
\end{remark}

\section{Surface area estimation based on the Devroye--Wise estimator}\label{dwsec}

\subsection{A conjecture on the Devroye--Wise estimator}\label{dwconj}

Since the set $S$ is in general unknown, we first propose the natural plug-in idea of computing 
$|\partial \hat{S}|_{d-1}$, where $\hat{S}$ is an estimator of $S$.
There are several kinds of set estimators, depending on the geometric restrictions imposed on $S$ and the structure of the data 
(see \cite{ch:16,DW} and references therein).
One of  the most studied in the literature, which is also universally consistent, is the Devroye--Wise  estimator (see \cite{DW}) that was introduced in \eqref{dw}. 
This all-purpose estimator has the advantage that it is quite easy   to compute the intersection of a line with its boundary, as follows: 
Given a line $r_{\theta,y}$, we can compute $\mathbb{Y}_i=\partial \B(X_i,\eps_n)\cap r_{\theta,y}$, and then
$\mathbb{Z}_i=\{y\in \mathbb{Y}_i,d(y,\X)\geq \eps_n\}$, so we have that, with probability one, 
$$\cup_i \mathbb{Z}_i=r_{\theta,y}\cap \partial \hat{S}_{\eps_n}(\X).$$
Indeed, suppose, on the contrary, that there exists a $z\in \cup_i \mathbb{Z}_i$ and $z\in \mathring{\hat{S}}_{\eps_n}$, then we have 
$d(z,\X)=\eps_n$ and $z\in \mathcal{H}\{X_i,d(X_i,z)=\eps_n\}$ (where $\mathcal{H}(E)$ is the convex hull of $E$). Thus, there are at least $d+1$ observations on the same hypersphere of given radius $\eps_n$, but this event has probability $0$ (see \cite{beatesis}).

We conjecture that the plug-in estimator $|\partial \hat{S}_{\eps_n}(\X)|_{d-1}$ satisfies the following:
\begin{enumerate}
	\item If $\eps_n<d_H(\X,S)$, then $\partial \hat{S}_{\eps_n}(\X)$ does not converge to $\partial S$ and 
	$|\partial \hat{S}_{\eps_n}(\X)|_{d-1}$ does not converge to $|\partial S|_{d-1}$. 
	\item If $\eps_n= d_H(\X,S)$, then $\partial \hat{S}_{\eps_n}(\X)$ converges to $\partial S$ with the best possible rate but 
	$|\partial \hat{S}_{\eps_n}(\X)|_{d-1}$ does not converge to $|\partial S|_{d-1}$ but greatly overestimates it.
	\item If $\eps_n\gg d_H(\X,S)$ and $\eps_n\rightarrow 0$, then $\partial \hat{S}_{\eps_n}(\X)$ converges to $\partial S$ and
	$|\partial \hat{S}_{\eps_n}(\X)|_{d-1}$ converges to $|\partial S|_{d-1}$
	but we can expect that the rate is greater than $\eps_n$ (namely $||\partial \hat{S}_{\eps_n}(\X)|_{d-1}-|\partial S|_{d-1}|\geq \mathcal{O}(\eps_n)\gg d_H(\X,S)$). Indeed,  if $S$ fulfills the 
	outside and inside rolling ball conditions, then, for $n$ large enough, we have that $B(S,\eps_n-d_H(\X,S))\subset \hat{S}_{\eps_n}(\X) \subset B(S,\eps_n)$, which in turn gives that $||\partial \hat{S}_{\eps_n}(\X)|_{d-1}-|\partial S|_{d-1}|\geq \mathcal{O}(\eps_n)\gg d_H(\X,S)$. 
	
\end{enumerate}

\subsection{A surface estimator based on the  Devroye--Wise estimator} \label{dwesti}

The aim of this section is to propose an estimator for the surface area based on the Devroye--Wise support estimator and Crofton's formula 
that can attain a convergence rate of order $d_H(\mathbb{X},S)$. 
The whole procedure is defined for any set $\mathbb{X}$, but is not necessarily finite because we will apply our estimator to the case in which $\mathbb{X}$ is the trajectory of a Brownian motion.
If $\mathbb{X}$ is not finite, then for a given $\varepsilon>0$, we write $\hat{S}_{\eps}(\mathbb{X})=B(\mathbb{X},\varepsilon)$. 
This procedure replaces $n_{\partial S}(\theta,y)$ by $\hat{n}_{\eps,\mathbb{X}}(\theta,y)$ introduced in Definition \ref{hatn}, 
and then integrates $\hat{n}_{\epsilon,\mathbb{X}}(\theta,y)$  as in Crofton's formula (see \eqref{defn1}).   We will prove that (see Remark \ref{rem0}) by the $(C,\eps_0)$-regularity of the boundary,
with probability one,  $r_{\theta,y}$ is not included in any $(d-1)$-dimensional affine tangent space (tangent to $\partial S$).
Then, $n_{\partial S}(\theta,y)=2k_S(\theta,y)$, where $k_S(\theta,y)$ is the number of connected components of $r_{\theta,y}\cap S$. 

\begin{definition}\label{hatn}
	Let $\eps$ be a positive real number  and $\mathbb{X}\subset S$ be a set (not necessarily finite).	Consider a line $r_{\theta,y}$. If $ \hat{S}_{\eps}(\mathbb{X}) \cap r_{\theta,y}=\emptyset$, then define $\hat{n}_{\eps,\mathbb{X}}(\theta,y)=0$. If not, then:
	
	\begin{itemize}
		\item denote by $I_1,\dots, I_m$ the connected components of $ \hat{S}_{\eps}(\mathbb{X})\cap r_{\theta,y}$.
		Order this sequence in such a way that $I_i=(a_i,b_i)$, with $a_1<b_1< \dots <a_m<b_m$.
		\item If for some consecutive intervals $I_{i},I_{i+1},\dots,I_{i+\ell}$, for all $a_i<\lambda<b_{i+\ell}$ and $t=y+\lambda\theta \in r_{\theta,y}$,      $d(t,\mathbb{X})\leq 4\eps$, define  $A_i=(a_i,b_{i+\ell})$.
		\item     Let $j$ be the number of disjoint open intervals $A_1,\dots,A_j$ that this process ended with.
		Then define $\hat{n}_{\eps,\mathbb{X}}(\theta,y)=2j$.
	\end{itemize}
\end{definition}

To roughly summarize this, we consider the connected components of $\hat{S}_{\eps}\cap r_{\theta,y}$ and then `link or glue' the ones that are in the same connected component of  $\hat{S}_{4\eps}\cap r_{\theta,y}$. In the sequel, we will refer to this process as the  gluing procedure.

To gain some insight into the relationship between $\hat{n}_{\eps,\mathbb{X}}(\theta,y)$ and $n_{\partial \hat{S}_{\eps}(\mathbb{X})}(\theta,y)$, observe that
$\hat{n}_{\eps,\mathbb{X}}(\theta,y)\leq n_{\partial \hat{S}_{\eps}(\mathbb{X})}(\theta,y)$.   
We also have that $\hat{n}_{\eps,\mathbb{X}}(\theta,y) \leq n_{\partial \hat{S}_{4\eps}(\mathbb{X})}(\theta,y)$.   Indeed, let 
$C_1,\ldots,C_K$ be the connected components of $r_{\theta,y}\cap \hat{S}_{4\eps}$ and note that:  
\begin{enumerate}
	\item For each $j$ there exists an index $i$ such that $I_j\subset C_i$.
	\item If $d(C_i,\mathbb{X})> \eps$ for all $j$, then we have that $I_j\cap C_i=\emptyset$. 
	\item If $d(C_i,\mathbb{X})\leq \eps$, then there exists an $I_j\subset C_i$ and all the $I_j$ such that $I_j\subset C_i$ are glued by the proposed procedure. Thus, there exists a unique $j'$ such that $A_{j'}\subset C_j$.
\end{enumerate}

Our first proposed estimator is
\begin{equation}\label{defn1}
	\hat{I}_{d-1}(\mathbb{X},\eps)=\frac{1}{\beta(d)} \int_{\theta \in {(\mathcal{S}^+)}^{d-1}}\int_{y\in \theta^\perp}  \hat{n}_{\eps,\mathbb{X}}(\theta,y)d\mu_{d-1}(y)d\theta.
\end{equation}

Under the assumption that $\partial S$ has a bounded number $N_S$ of linear intersections (see Definition \ref{boundinter}), we will consider, for a given $N_0\geq N_S$,
\begin{equation*}\label{defn1bis}
	\hat{I}^{N_0}_{d-1}(\mathbb{X},\eps)=\frac{1}{\beta(d)} \int_{\theta \in {(\mathcal{S}^+)}^{d-1}}\int_{y\in \theta^\perp}  \min (\hat{n}_{\eps,\mathbb{X}}(\theta,y),N_0)d\mu_{d-1}(y)d\theta.
\end{equation*}

\subsection{Main results on the Devroye--Wise based estimator.}\label{mainDW}

\begin{theorem} \label{mainth} Let  $S\subset \mathbb{R}^d$ be a compact set fulfilling the outside and inside $\alpha$-rolling conditions.
	Assume also that $S$ is $(C,\eps_0)$-regular for some positive constants $C$ and $\eps_0$.
	Let $\X=\{X_1,\dots,X_n\} \subset S$.
	Let $\eps_n \to 0$ be such that $d_H(\X, S)\leq \eps_n$. Then,  
	\begin{equation}\label{maintheq}
		\hat{I}_{d-1}(\mathbb{X}_n,\eps_n)=|\partial S|_{d-1}+\mathcal{O}(\sqrt{\eps_n}).
	\end{equation}
	Moreover, for $n$ large enough, 
	
	\begin{equation*}
		|\mathcal{O}(\sqrt{\eps_n})|\leq \frac{4C\text{diam}(S)}{3\beta(d)\sqrt{\alpha}}\sqrt{\eps_n}.
	\end{equation*}
\end{theorem}


The idea of the proof of Theorem \ref{mainth} consists of proving that our algorithm allows a perfect estimation
of $n_{\partial S}(\theta,y)$ for the lines that are `far enough'
(fulfilling $L(\eps)$ for some $\eps>0$) from the tangent spaces. For the rest of the lines, we will prove in Corollary \ref{coraux2} that, 
under $(C,\epsilon_0)$-regularity, the integral of $\hat{n}_{\eps_n,\X}(\theta,y)$ on the set of these lines, 
is bounded from above by $C'\eps_n^{1/2}$,  $C'$ being a positive constant. Roughly speaking, a line fulfilling condition 
$L(\eps)$ does not meet the estimator $\partial \hat{S}_{\epsilon_n}$ too many times.

From Theorem \ref{mainth} and Theorem 4 in \cite{crc:04}, we can obtain the rate of convergence for the iid case:

\begin{corollary} \label{mainth1} Let  $S\subset\mathbb{R}^d$ be a compact set fulfilling the inside and outside $\alpha$-rolling conditions.
	Assume also that $S$ is $(C,\eps_0)$-regular for some positive constants $C$ and $\eps_0$.
	Let $\X=\{X_1,\dots,X_n\}$ be the set of observations of an iid sample of $X$ with distribution $P_X$ supported on $S$.
	Assume that $P_X$ has density $f$ (w.r.t. $\mu_d$)
	bounded from below by some $c>0$.
	Let $\eps_n=C'(\ln(n)/n)^{1/d}$ and $C'>(6/(c\omega_d))^{1/d}$.
	Then, with probability one, for $n$ large enough,
	\begin{equation*}
		\hat{I}_{d-1}(\mathbb{X}_n,\eps_n)=|\partial S|_{d-1}+\mathcal{O}\left(\left(\frac{\ln n}{n} \right)^{\frac{1}{2d}} \right).
	\end{equation*}
	
\end{corollary}

As mentioned in Section 5.2 in  \cite{crc:04},  if $\eps_n=2 \max_i \min_{j\neq i} ||X_i-X_j||$, 
then with probability one, for $n$ large enough, $\eps_n\leq 2d_H(\X,S)$, which together with Corollary \ref{mainth1}, entails that, with the aforementioned choice for $\eps_n$, our proposal is fully data driven, for the iid case.\\

If the number of linear intersections  of $\partial S$ is assumed to be bounded by a constant $N_S$, the use of $\min(\hat{n}_{\eps_n},N_0)$ (for any $N_0\geq N_S$)
allows us to obtain better convergence rates.
\begin{theorem} \label{mainthbis} Let  $S\subset \mathbb{R}^d$  be a compact set fulfilling the outside and inside $\alpha$-rolling conditions.
	Assume also that $S$ is $(C,\eps_0)$-regular for some positive constants $C$ and $\eps_0$, and that the number of linear intersections of $\partial S$ is bounded by 
	$N_S$. 	Let  $\X =\{X_1,\dots,X_n\}\subset S$.
	Let $\eps_n \to 0$ be such that $d_H(\X, S)\leq \eps_n$ and $N_0\geq N_S$. Then,  
	\begin{equation*}
		\hat{I}^{N_0}_{d-1}(\mathbb{X}_n,\eps_n)=|\partial S|_{d-1}+\mathcal{O}(\eps_n).
	\end{equation*}
	
	Moreover, for $n$ large enough, $|\mathcal{O}(\eps_n)|\leq  4C(N_0+N_S)\eps_n/\beta(d).$
\end{theorem}

As before, we give the convergence rate associated to the iid setting and the RBM hypothesis as two corollaries
of Theorem \ref{mainthbis}.

\begin{corollary} \label{mainth1bis} Let  $S\subset\mathbb{R}^d$ be a compact set fulfilling the inside and outside $\alpha$-rolling conditions.
	Assume also that $S$ is $(C,\eps_0)$-regular for some positive constants $C$ and $\eps_0$, and that $\partial S$ has a bounded number of linear intersections. 
	Let $\X=\{X_1,\dots,X_n\}$ be the set of observations of an iid sample with distribution $P_X$, supported on $S$.
	Assume that $P_X$ has density $f$ (w.r.t. $\mu_d$)
	bounded from below by some $c>0$.
	Let $\eps_n=C'(\ln(n)/n)^{1/d}$ and $C'>(6/(c\omega_d))^{1/d}$.
	Then, with probability one, for $n$ large enough,
	\begin{equation*}
		\hat{I}^{N_0}_{d-1}(\mathbb{X}_n,\eps_n)=|\partial S|_{d-1}+\mathcal{O}\left(\left(\frac{\ln(n)}{n}\right)^{\frac{1}{d}}\right).
	\end{equation*}
\end{corollary}

Here again, the choice of $\eps_n=2 \max_i \min_j ||X_i-X_j||$ is suitable but now the price to pay is the selection of the parameter $N_0$.\\

In a more general setting, the conclusion of Theorem \ref{mainthbis} holds when the set of points $\X$ is replaced by the trajectory $\mathbb{X}_T$ of any stochastic process $\{X_t\}_{t>0}$ included in $S$, observed in $[0,T]$, such that $d_H(\mathbb{X}_T,S)\to 0$ as $T\to \infty$. Observe that the estimator $\hat{I}_{d-1}^{N_0}(\mathbb{X},\eps)$ is well defined, even when $\mathbb{X}_T$ is not a finite set (see Definition \ref{hatn}).  We will assume that $S$ is bounded with connected interior and $\partial S$ is $\mathcal{C}^2$.
This is the case (for example) of some reflected diffusions, and in particular the RBM.
This has recently been proven in Corollary 1 in  \cite{ch:16} for RBM without drift (see also \cite{ch:20} and \cite{ch:212} for the RBM with drift).
The definition of an RBM with drift is as follows: given a $d$-dimensional Brownian motion $\{B_t\}_{t\geq 0}$ departing from 
$B_0=0$ and  
defined on a filtered probability space $(\Omega,\mathcal{F},\{\mathcal{F}_t\}_{t\geq 0},\prob_x)$, an RBM with drift is the (unique) solution to the following stochastic differential equation on $S$:
\begin{equation*}
	X_t= X_0+ B_t-\frac{1}{2}\int_0^t\nabla_f(X_s)ds-\int_0^t\eta_{X_s}\xi(ds),
	\quad\text{ where } X_t\in \overline{D},\ \forall t\geq 0,
\end{equation*} 

\noindent where the drift, $\nabla_f (x)$, is given by the gradient of a function $f$ and is assumed to be Lipschitz, $\{\xi_t\}_{t\geq 0}$ is the corresponding local time; that is, a one-dimensional continuous non-decreasing process with $\xi_0=0$ that satisfies $\xi_t=\int_0^t\mathbb{I}_{\{X_s\in\partial S\}}d\xi_s$.  Since the drift is given by the gradient of a function and $S$ is compact, we have that its stationary distribution has a density bounded from below by a constant.

\begin{corollary}\label{maintdh2bis}  Let $S\subset \mathbb{R}^d$ be a non-empty compact set with connected interior such
	that $S=\overline{\mathring{S}}$, and suppose that $S$ fulfills  the outside and inside  $\alpha$-rolling conditions.
	Assume also that $S$ is $(C,\eps_0)$-regular for some positive constants $C$ and $\eps_0$ and that the number of linear intersections of $\partial S$ is bounded by 
	$N_S$. 	Let $\mathbb{X}_T\subset S$ be as before. 
	Then, with probability one, for $T$ large enough,
	\begin{equation*}
		\hat{I}_{d-1}^{N_0}(\mathbb{X}_T,\eps_T)=|\partial S|_{d-1}+o\left(\left(\frac{\ln(T)^2}{T}\right)^{\frac{1}{d}}\right),
	\end{equation*} 
	where $\eps_T=o((\ln(T)^2/T)^{1/d})$.
\end{corollary}

\subsection{The algorithm} \label{alg}

We will now describe an algorithm to compute  $\hat{n}_{\eps,\X}(\theta,y)$ for a given $(\theta,y)$, when the input is a finite set of $n$ elements and $\eps>0$.
For a reflected diffusion, we take $\X\subset \mathbb{X}_T$ to be a dense enough subset of $n$ points. Observe that this is not restrictive because $\mathbb{X}_T$ is stored as a finite set of points in a computer. 

\begin{enumerate} 
	\item[1.] For each $i$, compute   $d_i:=d(r_{\theta,y},X_i)=\sqrt{||X_i-y||^2-\langle X_i-y,\theta \rangle^2}$.
	\item[2.] Compute  the connected components $I_i$  of $r_{\theta,y}\cap \hat{S}_{\eps}(\X)$ according to the following steps:
	Initialize the list of the extremes of these intervals  by listz$=\emptyset$ and listl$=\emptyset$. Then,  
	for $i=1$ to $n$:
	\begin{itemize}
		\item If $d_i=\eps$, then $N_i=1$, $\ell_1=\langle X_i-y,\theta \rangle $ and $z_1=\mathcal{B}(X_i,\eps)\cap  r_{\theta,y}=y+\ell_1 \theta$
		\item  If $d_i<\eps$, then $N_i=2$ and compute $\ell_1=\langle X_i-y,\theta \rangle-\sqrt{\eps^2-d_i^2}$ and $\ell_2=\langle X_i-y,\theta \rangle+\sqrt{\eps^2-d_i^2}$. Then $z_1=y+\ell_1\theta $ and $z_2=y+\ell_2\theta $ such that $\{z_1,z_2\}=\mathcal{B}(X_i,\eps_n)\cap r_{\theta,y}$.
		\item For $j=1$ to $N_i$: if $d(z_j,\X)\geq \eps$, do listz=listz$\cup\{z_j\}$ and listl=listl$\cup\{\ell_j\}$.\\
		
		From the comments  at the beginning of subsection \ref{dwconj}, we know that, with probability one, listz equals $r_{\theta,y}\cap \partial \hat{S}_{\eps}$. 
		
		\item Sort listl. With probability one, listl has an even number, $2m$, of elements (see the comments at the beginning of subsection \ref{dwesti}),  and 
		define $a_i$ and $b_i$ such that $\ell_{2(i-1)+1}=a_i$, $\ell_{2i}=b_i$ (which corresponds to $a_i$ and $b_i$ in Definition \ref{hatn}; i.e. $(a_i,b_i)$ are the connected components of $r_{\theta,y}\cap \hat{S}_{\eps}(\mathbb{X})$).
	\end{itemize}
	
	\item[3] Obtain the $a'_i$ and $b'_i$ such that  $I'_i=(a_i',b'_i)$ are the connected components of $\hat{S}_{4\eps}(\X)\cap r_{\theta,y}$ by using the same procedure.
	
	\item[4.] Lastly, compute $\hat{n}_{\eps,\mathbb{X}}(\theta,y)$, as follows:\\
	initialization $\hat{n}_{\eps,\X}(\theta,y)=m$.\\
	
	For $i=1$ to $m-1$:\\
	
	\begin{itemize}
		\item If there exists $k$ such that $(b_i,a_{i+1})\subset I'_k$, then: 
		$$\hat{n}_{\eps,\X}(\theta,y)=\hat{n}_{\eps,\X}(\theta,y)-1.$$
	\end{itemize}
	
	\item[5.] $\hat{n}_{\eps,\X}(\theta,y)=2 \hat{n}_{\eps,\X}(\theta,y)$.

\end{enumerate}


\section{The approach based on the $\alpha$-convex hull}
\label{alphsec}
\subsection{The estimator based on the $\alpha'$-hull assuming the $\alpha$-rolling ball condition}
In \cite{alphshapr}, it was proven that in dimension two, under some regularity assumptions, the length of
the boundary of the $\alpha$-shape of an iid sample converges to the length of the boundary of the set. 
The $\alpha$-shape has the very good property that its boundary is very easy to compute, and hence so is its surface measure.
Unfortunately, we are unsure that the results can be extended to higher dimensions. Nevertheless, considering the $\alpha$-convex hull
(which is quite close to the $\alpha$-shape) allows us to extend the results on the surface measure to any dimension. 
The following deterministic theorem states that, for all  $0<\alpha'<\alpha$, the surface measure of the boundary of the $\alpha'$-convex hull $\X\subset S$  converges to $|\partial S|_{d-1}$ with a rate that depends on $d_H(\partial C_{\alpha'}(\X),\partial S)$.

\begin{theorem} \label{mainth2} Let  $S\subset\mathbb{R}^d$ be a compact set  such that $\partial S$ is a $(d-1)$-dimensional $\mathcal{C}^2$ manifold with reach $\alpha>0$.  Let $\alpha'<\alpha$ be a positive constant and 
	let $\X\subset S$ be a finite set such that  $d_H(\X,S)<\frac{1}{2}\frac{\alpha \alpha'}{\alpha+\alpha'}$  and $d_H(\partial C_{\alpha'}(\X),\partial S)\leq \eps_n$ with 
	$$\eps_n\leq \min\left\{\frac{\alpha \alpha'}{16(\alpha+\alpha')},\frac{1}{(d-1)\alpha}\right\}.$$ 
	Then,
	\begin{itemize}
		\item[1.] $\pi_{\partial S}:\partial C_{\alpha'}(\X)\rightarrow \partial S$ (where $\pi_{\partial S}(x)$ denotes the projection onto $\partial S$)
		is one to one, and  
		\item[2.] $\big||\partial S|_{d-1}-|\partial C_{\alpha'}(\X)|_{d-1}\big|\leq (d-1)\left(\frac{3}{2}\alpha+32\frac{\alpha+\alpha'}{\alpha \alpha'} \right)\eps_n (1+o(1))$.  
	\end{itemize}
\end{theorem}

As previously, we can deduce the convergence rates from the deterministic theorem and results in \cite{alphshapr} under the iid assumption. 

\begin{corollary} \label{mainth2iid} Let  $S\subset\mathbb{R}^d$ be a compact set such that $\partial S$ is a $(d-1)$-dimensional $\mathcal{C}^2$ manifold with reach $\alpha>0$.   Let $\X=\{X_1,\ldots,X_n\}$ be an iid sample of $X$ with distribution $P_X$ supported on $S$.
	Assume that $P_X$ has density $f$ (w.r.t. $\mu_d$)
	bounded from below by some $c>0$.
	Suppose $\alpha'<\alpha$. Then, with probability one, for $n$ large enough,
	$$\big||\partial S|_{d-1}-|\partial C_{\alpha'}(\X)|_{d-1}\big|=\mathcal{O}\left(\left(\frac{\ln(n)}{n}\right)^{\frac{2}{d+1}}\right).$$
\end{corollary}

In this case we do not need  the additional hypothesis of $(C,\eps_0)$-regularity. The convergence rate is far better than the one 
given in Theorem \ref{mainth},  where the price to pay is the computational cost  when $d$ increases.  Indeed, as detailed in next section, the computation of the $\alpha$-convex hull requires us to start by the computation of the Delaunay complex. With regard to the parameter selection  $\alpha'$, a fully
data driven (but computationally expensive) method is proposed in \cite{choosealpha}.

\subsection{Computation with the use of Crofton's formula}
Unfortunately, the explicit computation of $|\partial C_{\alpha}(\X)|_{d-1}$   is very difficult.   However, from the results in Lemma \ref{rconvhullgeo}, we derive that 	
we can make use Crofton's formulae and the Monte Carlo method to estimate $|\partial C_{\alpha}(\X)|_{d-1}$.  
This, as we will see, is based on the fact that the computation of $\check{n}_\alpha(\theta,y):= \ n_{\partial C_\alpha(\X)}(\theta,y)$ is feasible.
It first requires the computation of the  $\alpha$-convex hull, as well as the convex hull, of $\X$.
Recall that the convex hull $\mathcal{H}(\X)$ of $\X$ is equal to the intersection of a finite number of half-spaces 
$\mathcal{H}(\X)=\bigcap_{i=1}^K H_i$ with $H_i=\{x\in \mathbb{R}^d, \langle x-y_i,u_i\rangle \leq 0\}$ for some $\{y_1,\ldots,y_K\}\subset \R^d$ and $\{u_1,\ldots,u_K\}\subset \mathcal{S}^{d-1}$.
\begin{figure}[h!]\label{pts}
	\begin{center}
		\includegraphics[scale=0.5]{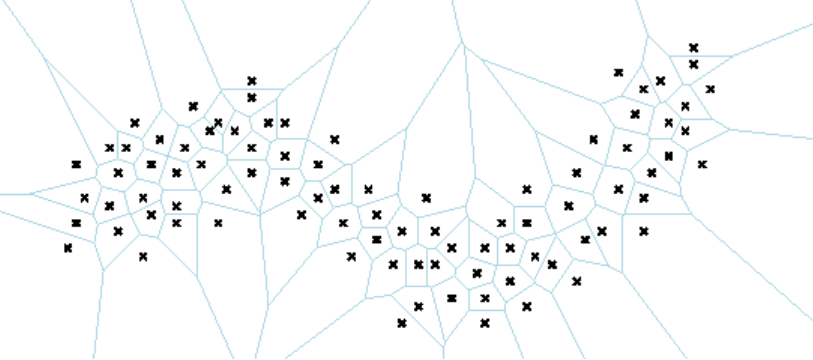}
		\caption{Points in $\mathbb{R}^2$ and the associated Voronoi diagram.}
	\end{center}
\end{figure}

In  \cite{edel}, it is proved for dimension $2$ that $C_\alpha(\X)^c$ is the union of a finite number of balls and the aforementioned half-spaces but mentioned that the generalization is not difficult.
The centres $O_i$ of these balls and  their radii $r_i$ are  obtained by computing the Delaunay complex. The computational cost of the Delaunay complex is the main part of the computational cost of our algorithm, which is defined as follows:  
\begin{enumerate}
	\item Compute all the Delaunay simplices $\sigma_i=\mathcal{H}(\{X_{i_1},\ldots,X_{i_{d+1}}\})$; that is, those such that 
	$\mathring{\B}(O_{i},r_{i})\cap \X=\emptyset$ and  $\partial \B(O_{i},r_{i})$ is the sphere circumscribed to $X_{i_1},\ldots,X_{i_{d+1}}$.
	
	\begin{figure}[h!]\label{Cr1}
		\begin{center}
			\includegraphics[scale=0.5]{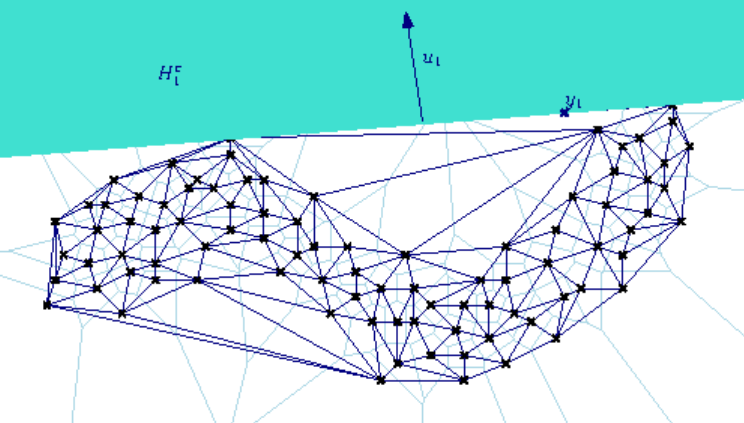}
			\caption{Points in $\mathbb{R}^2$, The associated Delaunay complex and an half space $H_1^c$ }
		\end{center}
	\end{figure}
	\item Sort the indices so that $r_i$ are decreasing, and define $K'=\#\{r_i,r_i\geq \alpha'\}$.
	\item Define $B_i^+=\mathring{\mathcal{B}}(O_i,r_i)$ for $i\in \{1,\ldots,K'\}$. Clearly, $r_i\geq \alpha'$ for all $i=1,\dots, K'$.
	
	\begin{figure}[h!]\label{Cr2}
		\begin{center}
			\includegraphics[scale=0.5]{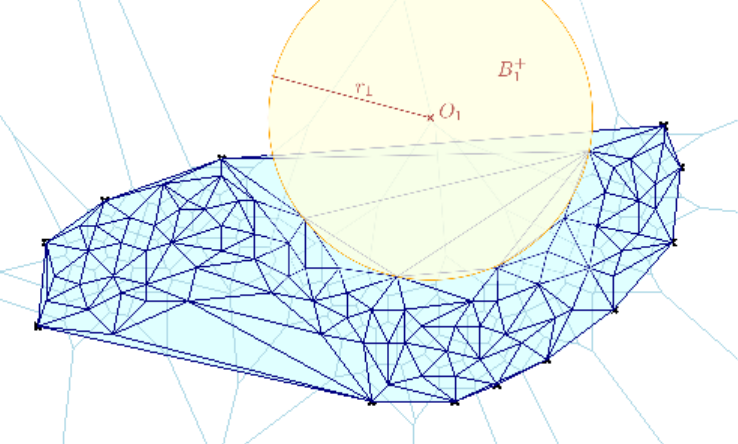}
			\caption{The convex hull of the points (blue) and a ball $B_1^+$ }
		\end{center}
	\end{figure}
	
	\begin{figure}[h!]\label{Cr3}
		\begin{center}
			\includegraphics[scale=0.5]{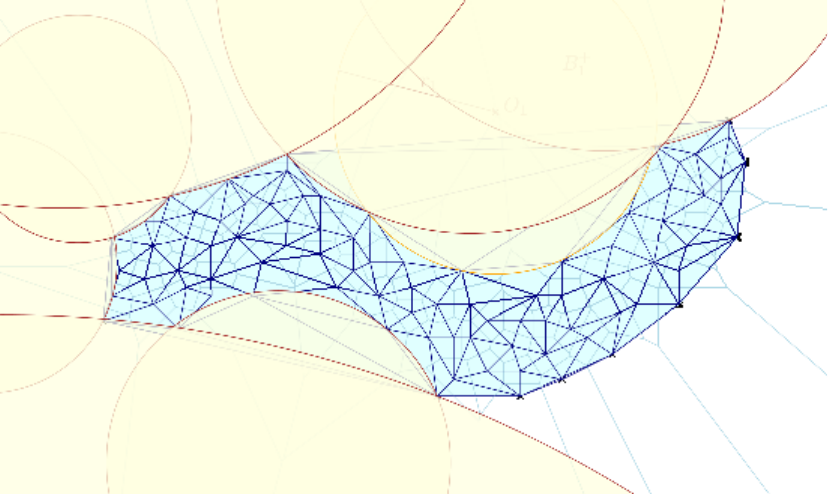}
			\caption{The convex hull of the points (blue) and all the $B_i^+$}
		\end{center}
	\end{figure}
	\item Compute the faces of the boundary of the $\alpha'$-shape (see \cite{edel}), which are the $f_i=\mathcal{H}(\{X_{i_1},\ldots,X_{i_d}\})$ such that 
	there exists a unique  $j\geq K'+1$ such that $f_i\subset \sigma_j$.\\
	
	\begin{figure}[h!]\label{Cr4}
		\begin{center}
			\includegraphics[scale=0.5]{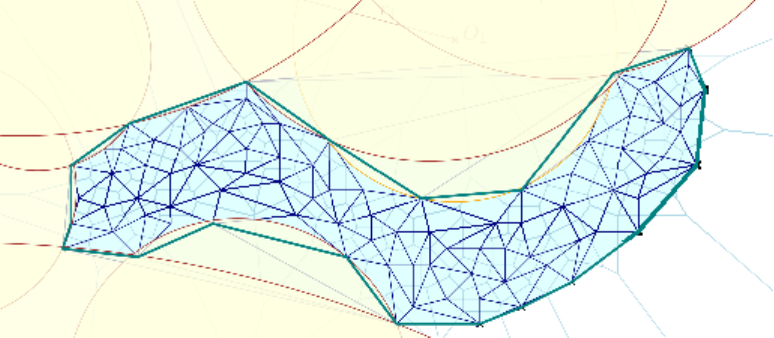}
			\caption{The convex hull of the points, all the $B_i^+$ and the boundary faces (green)}
		\end{center}
	\end{figure}
	
	Also compute $\Omega_i$ (resp. $\rho_i$), which is the center (resp. radius) of the sphere circumscribed to
	$X_{i_1},\ldots,X_{i_d}$ in the plane spanned by $X_{i_1},\ldots,X_{i_d}$. \\
	
	Now we have two different cases:
	\begin{enumerate}
		\item $f_i$ is a face of $\partial \mathcal{H}(\X)$; that is, there exists $j'$ such that $f_i\subset H_{j'}$. Then, define $w_i=u_{j'}$.
		\item $f_i$ is not a face of $\partial \mathcal{H}(\X)$, thus there exists $j'\leq K'$ such that  $f_i\subset \sigma_{j'}$.  Then, define $w_i=(O_{j}-O_{j'})/||O_{j}-O_{j'}||,$ with 
		$j\geq K+1$ such that $f_i\subset \sigma_j$.
	\end{enumerate}
	
	\begin{figure}[h!]\label{Cr5}
		\begin{center}
			\includegraphics[scale=0.5]{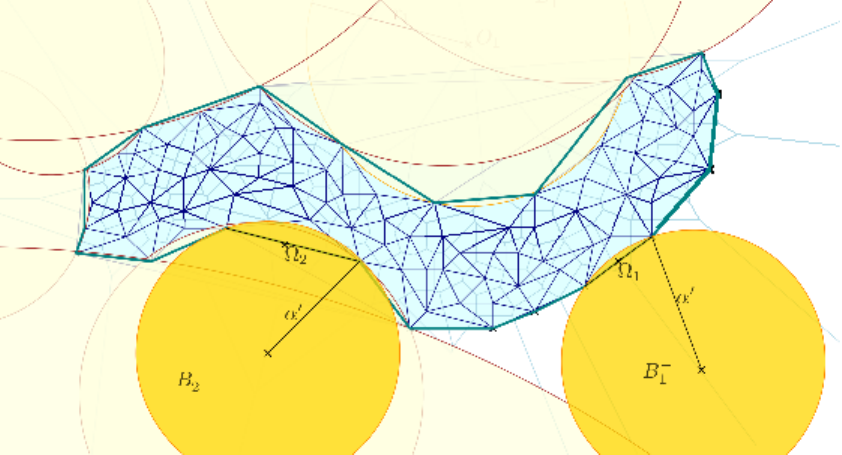}
			\caption{The convex hull of the points, all the $B_i^+$, the boundary faces (green) and two $B^-$. $B_1^-$ correspond to case $(a)$ and $B_2^-$ corresponds to case $(b)$.}
		\end{center}
	\end{figure}
	
	Define $B_i^{-}=\mathring{\mathcal{B}}(\Omega_i+\sqrt{\alpha'^2-\rho_i^2}w_i,\alpha')$. Then,
\end{enumerate}
\begin{equation}\label{rhull}
	C_{\alpha'}(\X)^c=\left(\bigcup_i H_i^c\right) \cup \left(\bigcup_iB_i^-\right)\cup \left(\bigcup_iB_i^+\right).
\end{equation}
To simplify notation, we write $C_\alpha(\X)^c=\bigcup_i B_i$.
Observe that if the line $r_{\theta,y}$ is chosen at random (w.r.t. $d\mu_{d-1} d\theta$),
with probability one, then we have 
$r_{\theta,y}\cap \partial B_i$, which contains less than three points.
\\

Initialize list=$\emptyset$. Then,

\noindent for each $i$,
\begin{itemize}
	\item compute $r_{\theta,y}\cap \partial B_i$.
	\item For all $z\in r_{\theta,y}\cap \partial B_i$, if for all $j$ $z\notin \mathring{B}_j$, then do list=list$\cup\{z\}$.
\end{itemize}
\noindent then $\check{n}(\theta,y)=\#$list.

\section{Discussion of the rates of convergence}
\label{rates}
In Corollary \ref{mainth2iid}, we obtained the same convergence rate as the one provided in  \cite{alphshapr} for $d=2$, which is conjectured as 
suboptimal. 
As mentioned in \cite{alphshapr}, if the measure of the symmetric difference between $S$ and an estimator $\hat{S}_n$ is bounded by $\eps_n$, then we can only expect that plug-in methods allow us to estimate
$\vert \partial S\vert_{d-1}$ 
with a convergence rate $\eps_n$. 
Thus, in the iid setting, the estimator defined by (6) (respectively (7) to (9)) can be seen as `optimal' relative to the use of the Devroye--Wise support estimator (respectively, the $\alpha$-convex hull support estimator) because they achieve the best possible convergence rates for those estimators. 
This is nevertheless far from being optimal: the minimax rate is conjectured to be $n^{-\frac{d+3}{2d+2}}$, 
which is the minimax rate for the volume estimation problem (see \cite{ar17}), and in \cite{kim} it is proved that the minimax rate
is the same for the volume estimation problem  and the surface area estimation problem (at least in the image setting, which usually extends to the iid setting).
Unfortunately, attaining this optimal rate for the surface area estimation problem  is much more involved, even in the easier
setting with data uniformly drawn in $S$ and  $S^c$ with perfect identification. No estimator attaining this rate has yet been proposed.

\section{Integralgeometric estimations via a Monte Carlo method and numerical experiments}\label{algo}

To estimate the surface area with a Monte Carlo method, we propose the following classical procedure.  
Generate a random sample
$\theta_1,\ldots,\theta_k$ that is uniformly distributed on $(\mathcal{S}^+)^{d-1}$.  For each $i=1,\dots,k$, 
draw a random sample $\aleph_i=\{y^i_1,\dots,y^i_{\ell}\}$ that is uniformly distributed on $[-L,L]^{d-1}\subset \theta_i^\perp$, independent of $\theta_1,\dots,\theta_k$, where $L=\max_{j=1,\dots,n} ||X_j||$. Then, the estimators are given by 
\begin{align}
	\hat{\hat{I}}_{d-1}^{({\ell},k)}(\partial S)=&\frac{(2L)^{d-1}}{\beta(d)}  \frac{1}{\ell k}\sum_{i=1}^k \sum_{j=1}^{\ell}\hat{n}_{\eps_n,\X}(\theta_i,y^i_j) \label{estim1} \\
	\hat{\hat{I}}_{d-1}^{({\ell},k,N_0)}(\partial S)=&\frac{(2L)^{d-1}}{\beta(d)}  \frac{1}{\ell k}\sum_{i=1}^k \sum_{j=1}^{\ell}\min(\hat{n}_{\eps_n,\X}(\theta_i,y^i_j),N_0) \label{estim1bis} \\
	\check{\check{I}}_{d-1}^{({\ell},k)}(\partial S)=&\frac{(2L)^{d-1}}{\beta(d)} \frac{1}{\ell k}\sum_{i=1}^k \sum_{j=1}^{\ell}\check{n}_r(\theta_i,y^i_j). \label{estim2} 
\end{align}

\section{Simulation study}
\label{simu}
The performance of \eqref{estim1} and \eqref{estim2} is illustrated through a simulation study.
We consider the sets 

$S(d,r)=\mathcal{B}_d(O,1)\setminus \mathring{\mathcal{B}}_d(O,r)$ for $d=2,3$, $r=0.5,0.6,0.7,0.8$ and $0.9$. 

On each set, we draw  $n=50,100,200,500,1000,2000$ and $4000$   iid random vectors supported on $S(d,r)$  , whose common distribution is $X=RZ$, where $R$  is a real valued 
random variable uniformly distributed on  $[1-r,1]$ and $Z$ is a random vector (independent of $R$) that is supported on the  $(d-1)$-dimensional   sphere.

For \eqref{estim1}, we computed the parameter $\eps_n$ as follows: for each sample point we calculate the distance to its 
closest point in the sample, and we   choose  $\eps_n$ as the third quantile of these $n$ distances.  
For \eqref{estim2}, we estimated the parameter $\alpha$ with the data-driven estimator proposed in \cite{choosealpha}. 
Roughly speaking, ``the largest value of $\alpha$ compatible with the $\alpha$-convexity assumption'' is chosen.

We choose $k=4000$ and $\ell=1$, at equation \eqref{estim1}  and the same for \eqref{estim2}.

To illustrate the convergence without the bias of the Monte Carlo step we compare our estimator with the Crofton based surface area
estimation on the true (unknown) set based on the same line sample.
More precisely, for each example (given by a dimension $d$, a radius $r$, a sample size $n$ and an experiment number $i$) we draw 
$\X$ as previously explained and $4000$ values of $(\theta_j,y_j)$ (i.e $4000$ lines) and then compute :

\begin{equation}\label{dwest}
	E_{i}^{DW}(d,n):=\frac{\sum_{j=1}^{4000} \left(\hat{n}_{\eps,\X}(\theta_j,y_j)-n_{\partial S_r}(\theta_j,y_j)\right)}
	{\sum_{j=1}^{4000} n_{\partial S_r}(\theta_j,y_j)} 
\end{equation}
and 
\begin{equation}\label{alfhull}
	E_{i}^{CH}(d,n):=\frac{\sum_{j=1}^{4000} \hat{n}_{\partial C_{\alpha}(\X)}(\theta_j,y_j)-n_{\partial S_r}(\theta_j,y_j)}
	{\sum_{j=1}^{4000} n_{\partial S_r}(\theta_j,y_j)}
\end{equation}

In Figure \ref{fig31} we show, for each $d$ and $r$, 
the results of the proposed method based on 57 experiment replications.
Black curves represent results for r-convex hull based surface area estimator.  
We present here the evolution of the extremal values of the error given by \eqref{alfhull} and \eqref{dwest} (dots), the $5\%$ and $95\%$ quantiles (dashed),  the $25\%$ , $50\%$ and the $75\%$ quantiles (plain).
The convergence towards $0$ (blue line) can be observed.
In red we present the same curves for the case of  the Devroye-Wise based estimator.
As expected due to theoretical results, convergence is quicker for the r-convex hull estimator than for the Devroye-Wise based surface area estimator (same curves, in red). 
This is particularly clear when $r_0\geq 0.7$.

\begin{figure}[htb] \centering \includegraphics[scale=0.45]{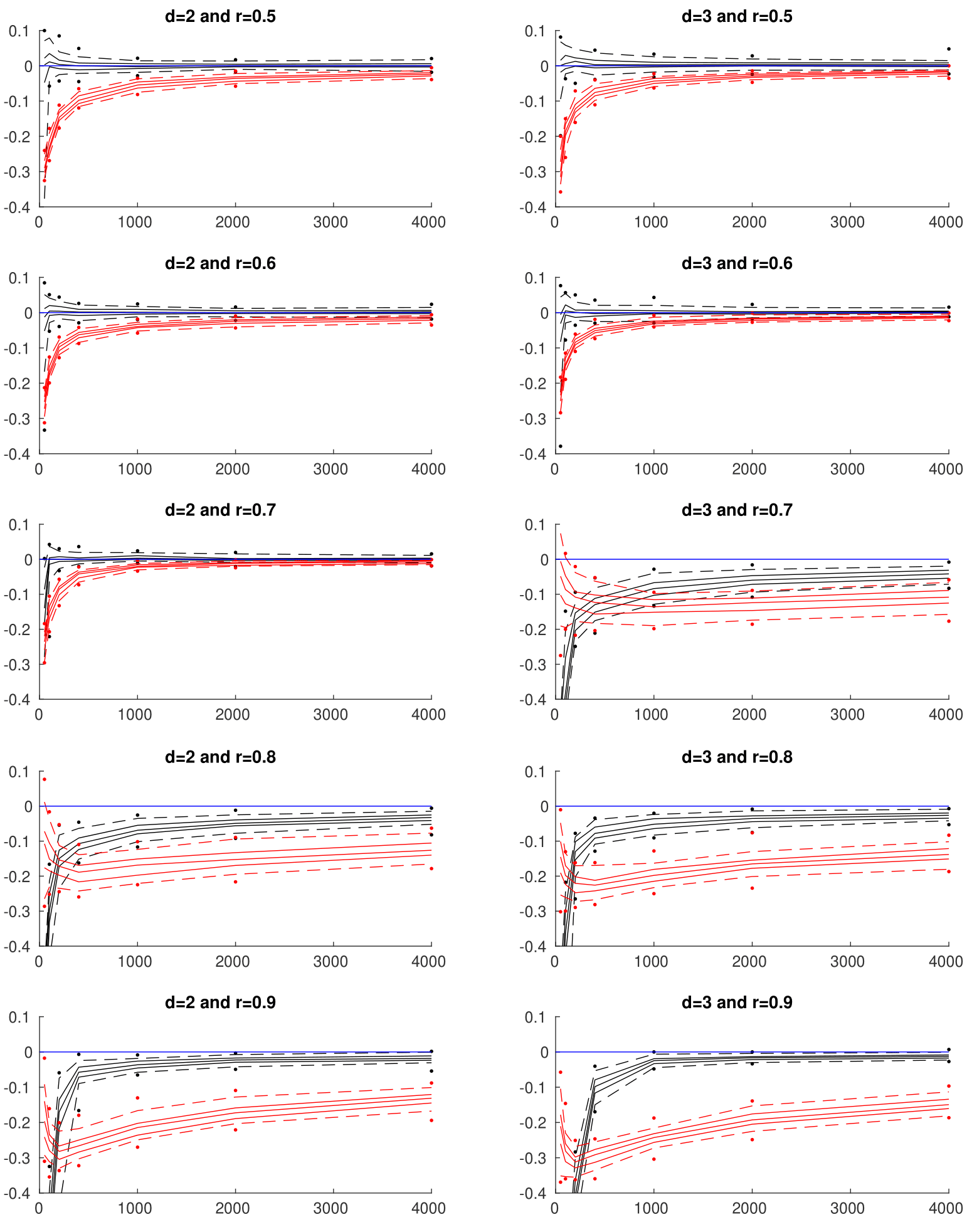} 
	\caption{We show in each panel results based on 57 replications of \eqref{dwest} (red) and \eqref{alfhull} (black), for different values of $d$ and $r$. Varying $n\in \{50,100,200,500,1000,2000,4000\}$. Dots are extremal values, dashed $5\%$ and $95\%$ quantiles, and plain lines $25\%$, $50\%$ (median) and $75\%$ quantiles. }
	\label{fig31} \end{figure}

\section{Appendix}

\subsection{Proofs of Theorems \ref{mainth} and \ref{mainthbis}}

\subsubsection*{Sketch of the proofs of Theorems \ref{mainth} and \ref{mainthbis}}

The idea is to consider separately  two subsets of  the set of lines that intersect $\partial \hat{S}_{\eps_n}(\X)$:

\begin{enumerate}
	\item If a line $r_{\theta,y}=y+\R\theta$ is `far enough' (fulfilling condition $L(\eps)$ for some $\eps>0$, see Definition \ref{A}) 
	from the tangent spaces, then our algorithm allows a perfect estimation of $n_{\partial S}(\theta,y)$, see Lemma \ref{propaux}.
	\item Considering the set of lines that are not `far enough' from the tangent spaces (denoted by $\mathcal{A}_{\eps_n}(\theta)$),
	see Definition \ref{A}),  Corollary \ref{coraux2} states that, under $(C,\epsilon_0)$-regularity, the integral of $\hat{n}_{\eps_n,\X}(\theta,y)$  on $\mathcal{A}_{\eps_n}(\theta)$ is
	bounded from above by $C'\eps_n^{1/2}$, where $C'$ is a positive constant.
	Theorem \ref{mainthbis} states that the previous bound can be improved to $C'\eps_n$, under $(C,\epsilon_0)$-regularity, if $\partial S$
	has a bounded number of linear intersections.
\end{enumerate}

\subsubsection{Condition   $L(\eps)$ }

We now define the two sets of lines to be tackled separately.
The lines that are `far' from an affine tangent space and the lines that are `close to being tangent' to $\partial S$.
More precisely, recall that  the unit outer normal vector $\eta_x$ at $x$ is well defined under the rolling ball hypothesis
(see Remark \ref{rem1}). 
Now we define 
\begin{equation*}
	\mathcal{T}_S=\{x+(\eta_x)^{\perp}: \ x\in \partial S\},
\end{equation*}
the  collection  of all the affine $(d-1)$-dimensional  tangent  spaces.

\begin{definition}\label{A}
	Let $\eps\geq 0$. A line $r_{\theta,y}=y+\mathbb{R}\theta$ fulfills \textbf{condition $L(\eps)$} if $y$ is at a distance
	larger than  $4\eps$ from all the affine
	hyper-planes $w+\eta^\perp \in \mathcal{T}_S$  satisfying $\langle \eta,\theta\rangle=0$;  
	that is, for all $x\in \partial S$ such that $\langle \eta_x,\theta \rangle=0$  we have that $d(y,x+\eta_x^{\perp})> 4 \eps$

	For a given $\theta$, we define
	\begin{equation*}
		\mathcal{A}_{\eps}(\theta)=\Big\{y \in \theta^{\perp}: ||y|| \leq diam(S) \text{ and } r_{\theta,y} \text{ does not satisfy } L(\eps)\Big\}.
	\end{equation*}
\end{definition}

\begin{remark}\label{rem2}
	Notice that
	$\ell_{\theta}(y)=\min_{x\in E_{\theta}(\partial S)}d(y,x+\eta_x^{\perp})$ is well defined because $E_{\theta}$ is compact and $x\mapsto \eta_x$
	is a continuous function, due to the regularity of $\partial S$. Moreover, if $y\in \theta^{\perp}$, then  $\ell_{\theta}(y)=d(y,F_{\theta})$; and consequently
	for all $t\in(0,d(y,F_{\theta})/4)$, $r_{\theta,y}$ satisfies the condition $L(d(y,F_{\theta})/4-t)$.
\end{remark}

\subsubsection{Some useful lemmas}

\begin{lemma} \label{inter} Let $S$ be a compact set fulfilling the outside and inside $\alpha$-rolling conditions.
	Let $r_{\theta,y}$ be a line that fulfills condition $L(0)$ and $r_{\theta,y}\cap \partial S\neq\emptyset$. Then, $r_{\theta,y}$ intersects $\partial S$ in a finite number of points.
\end{lemma}

\begin{proof} Because $S$ fulfills the outside and inside $\alpha$-rolling conditions, Theorem 1  in \cite{wal99} implies that for any $x\in \partial S$, the affine $(d-1)$-dimensional tangent space $T_x\partial S$ exists.
	If $r_{\theta,y}$ fulfills $L(0)$, then  $r_{\theta,y}$ is not included in any hyper-plane tangent to $S$.
	Suppose that $\partial S\cap r_{\theta,y}$ is not finite. Then, by compactness, one
	can extract a subsequence $t'_n\in \partial S\cap r_{\theta,y}$ that converges to $y'\in \partial
	S\cap r_{\theta,y}$.
	\begin{enumerate}
		\item Because $t_n'$ and $y'$ are in $r_{\theta,y}$, we have that, for all $n$, $(t'_n-y')/||t'_n-y'||=\pm\theta$.
		\item Because $t_n'$ and $y'$ are in $\partial S$, which is a $(d-1)$-dimensional $\mathcal{C}^1$ manifold (see Theorem 1 in \cite{wal99}), and $t_n'\rightarrow y'$, we have $\lim_{n\rightarrow +\infty} (t'_n-y')/||t'_n-y'||\in  T_{y'}\partial S$ (see Definition 4.3 in \cite{fed:56}).  
	\end{enumerate}
	These two facts imply that $\theta \in T_{y'}\partial S$,
	which contradicts the assumption that $r_{\theta,y}$ is not included in any hyper-plane tangent to $S$.
\end{proof}

\begin{lemma}\label{lemaux} Let $S\subset \mathbb{R}^d$ be a compact set
	fulfilling the outside and inside $\alpha$-rolling conditions.
	Let $\eps>0$ be such that  $\eps<\alpha/4$ and 	$\nu=2[2\eps(\alpha-2\eps)]^{1/2}$.
	For any line  $r_{\theta,y}$ fulfilling condition $L(\eps)$ and $r_{\theta,y}\cap \partial S\neq\emptyset$, we have that 
	$r_{\theta,y}$ meets $\partial S$ at a finite number of points
	$t_1,\dots,t_k$, where $t_{i+1}-t_i\geq 2\nu$ for all $i=1,\dots,k-1$.
	Consequently, if $\eps<\alpha/4$, then $k=\# (r_{\theta,y}\cap \partial S) \leq \text{diam}(S)\eps^{-1/2}/(4\sqrt{\alpha})$.
\end{lemma}

\begin{proof} If a line fulfills condition $L(\eps)$, then it fulfills condition $L(0)$. Consequently, the fact that $r_{\theta,y}$ intersects $\partial S$ in a finite number of points follows from Lemma \ref{inter}.
	Let us denote by $t_1<\dots <t_k$ the intersection of $r_{\theta,y}$ with
	$\partial S$.
	
	Let us denote by $\eta_{t_i}$ and $\eta_{t_{i+1}}$
	the outer normal vectors at $t_i$ and $t_{i+1}$, respectively.
	We have
	two cases: the open interval $(t_i,t_{i+1})\subset S^c$ or 	$(t_i,t_{i+1})\subset \mathring{S}$.
	Let us consider the first case (the proof
	for the second one is similar).
	
	Because $(t_i,t_{i+1})\subset \overline{S^c}$ and $S$ fulfills the inside $\alpha$-rolling condition on $t_i$, there exists a $z\in S$ such that $t_i\in
	\partial \mathcal{B}(z,\alpha)$ and $\mathcal{B}(z,\alpha)\subset S$.
	In
	particular, $\mathcal{B}(z,\alpha)\cap(t_i,t_{i+1})=\emptyset$, which implies
	$\langle \eta_{t_i}, \theta\rangle \geq 0$.
	
	Reasoning in the same way
	but with $t_{i+1}$,  we get $\langle \eta_{t_{i+1}}\theta\rangle \leq 0$.
	Given that
	$r_{\theta,y}$ is not included in any tangent hyperplane, we have that  $\langle
	\eta_{t_i}, \theta\rangle > 0$ and $\langle \eta_{t_{i+1}},
	\theta\rangle < 0$.
	
	If, for some $i$, $t_{i+1}-t_i<2\nu$, then, by Theorem
	3.8 in \cite{cm10},  there exists a curve $\gamma:[0,1]\to \partial S$ such
	that $\gamma(0)=t_i$, $\gamma(1)=t_{i+1}$ and
	$d(\gamma(t),r_{\theta,y})\leq 4\eps $ for all $t$. 	
	We also have the inside and outside $\alpha$-rolling conditions, which implies that $x\mapsto \eta_x$ is Lipschitz (see Theorem 1 in \cite{wal99}).
	From 
	$\langle \eta_{t_i}, \theta\rangle > 0$ and $\langle \eta_{t_{i+1}},
	\theta\rangle < 0$, it follows that there exists a $s_0\in (0,1)$ such that
	$\langle \eta_{\gamma(s_0)}, \theta\rangle=0$, which contradicts the hypothesis that $y$ is at
	a distance larger than $4\eps$ from all the  $(d-1)$-dimensional
	hyperplanes tangent to $S$.
	This proves that $t_{i+1}-t_i\geq 2\nu$ for all
	$i=1,\dots,k-1$.
\end{proof}

\begin{lemma}\label{boundint1}
	Let $S\subset \mathbb{R}^d$ be a compact set fulfilling the outside and inside $\alpha$-rolling conditions, with
	a $(C,\eps_0)$-regular boundary. Then, for all $\eps \leq \min\{\eps_0,\alpha\}/4$,
	
	\begin{equation*}\label{boundint1eq1}
		\int_{\theta \in {(\mathcal{S}^+)}^{d-1}} \int_{y\in \mathcal{A} _{\eps}(\theta)} n_{\partial S}(\theta,y) d\mu_{d-1}(y)d\theta \leq 2C \frac{\text{diam}(S)}{\sqrt{\alpha}}\sqrt{\eps}.
	\end{equation*}
	
	Moreover, if $\partial S$ has bounded number of linear intersections, then
	\begin{equation}\label{boundint1eq2}
		\int_{\theta \in {(\mathcal{S}^+)}^{d-1}} \int_{y\in \mathcal{A} _{\eps}(\theta)} n_{\partial S}(\theta,y) d\mu_{d-1}(y)d\theta \leq 4C N_S \eps.
	\end{equation}    
	
\end{lemma}

\begin{proof}  Observe that 
	\begin{multline*}
		\int_{\theta \in {(\mathcal{S}^+)}^{d-1}} \int_{y\in \mathcal{A} _{\eps}(\theta)} n_{\partial S}(\theta,y) d\mu_{d-1}(y) d\theta  \\
		= \int_{\theta \in {(\mathcal{S}^+)}^{d-1}} \int_{\ell=0}^{4\eps} \int_{\{y\in \theta^\perp:d(y,F_{\theta})=\ell\}} n_{\partial S}(\theta,y) d\mu_{d-2}(y)d\ell  d\theta.
	\end{multline*}
	According to Remark \ref{rem2}   if $y\in \theta^\perp:d(y,F_{\theta})=\ell$,   then, for all $t\in(0,\ell)$, $r_{\theta,y}$ fulfills 
	$L(\ell/4-t)$. From the proof of the previous lemma,
	it follows that for any $y\in \theta^{\perp}$ with $d(y,F_{\theta})=\ell$ and $\ell<4\eps$, 
	and any $t\in(0,\ell/4)$ 
	$$n_{\partial S}(\theta,y)\leq \text{diam}(S)(\ell/4-t)^{-1/2}/(4\sqrt{\alpha}).$$
	Hence, with $t\rightarrow 0$ we obtain 
	$n_{\partial S}(\theta,y)\leq \text{diam}(S)(\ell)^{-1/2}/(2\sqrt{\alpha})$, from which:

	\begin{align*}
		\int_{\theta \in {(\mathcal{S}^+)}^{d-1}} \int_{\ell=0}^{4\eps} \int_{\{y\in \theta^\perp:d(y,F_{\theta})=\ell\}} n_{\partial S}(\theta,y) d\mu_{d-2}(y)d\ell  d\theta& =  \\	
		& \hspace{-7cm}  \leq \int_{\theta \in {(\mathcal{S}^+)}^{d-1}} \int_{\ell=0}^{4\eps} \int_{\{y\in \theta^\perp: d(y,F_{\theta})=\ell\}} \frac{1}{  2   }\text{diam} (S)(\alpha \ell)^{-1/2} d\mu_{d-2}(y)d\ell d\theta\\
		& \hspace{-7cm} \leq  \int_{\theta \in {(\mathcal{S}^+)}^{d-1}} \int_{\ell=0}^{4\eps}\frac{1}{  2   }\text{diam} (S)(\alpha \ell)^{-1/2}  \int_{\{y\in\theta^\perp d(y,F_{\theta})=\ell\}} d\mu_{d-2}(y)d\ell  d\theta\\
		& \hspace{-7cm} \leq  \int_{\theta \in {(\mathcal{S}^+)}^{d-1}} \int_{\ell=0}^{4\eps}\frac{1}{  2   }\text{diam} (S)(\alpha \ell)^{-1/2} |\big\{y\in \theta^\perp: d(y,F_{\theta})=\ell\big\}|_{d-2}d\ell  d\theta.
	\end{align*}
	By the definition of $\varphi_\theta$,
	\begin{equation*}
		\Big|\big\{y\in \theta^\perp: \ell \leq d(y,F_{\theta})\leq \ell +d\ell  \big\}\Big|_{d-1}=\varphi_{\theta}(\ell +d\ell )-\varphi_{\theta}(\ell ).
	\end{equation*}
	From the $(C,\eps_0)$-regularity of $\partial S$ and the mean value theorem we obtain
	\begin{equation*}
		\Big|\big\{y\in \theta^\perp: d(y,F_{\theta})=\ell\big\}\Big|_{d-2}\leq \sup_{\eps\in (0,4\eps_0)} \varphi_\theta'(\eps)\leq C,
	\end{equation*}
	which implies
	\begin{multline*}
		\int_{\theta \in {(\mathcal{S}^+)}^{d-1}} \int_{y\in \mathcal{A} _{\eps}(\theta)} n_{\partial S}(\theta,y)d\mu_{d-1}(y)d\theta 
		\leq \\\int_{\theta \in {(\mathcal{S}^+)}^{d-1}} \int_{\ell=0}^{4\eps}  C \frac{1}{  2   }\text{diam} (S)(\alpha \ell)^{-1/2} d\ell d\theta
		\leq   2   C  \frac{\text{diam}(S)}{\sqrt{\alpha}}\sqrt{\eps}.
	\end{multline*}
	
	By applying exactly the same reasoning, under the hypothesis of the boundedness of the number of linear intersections for $\partial S$, we get
	\begin{equation*}
		\int_{\theta \in {(\mathcal{S}^+)}^{d-1}} \int_{y\in \mathcal{A} _{\eps}(\theta)} n_{\partial S}(\theta,y)d\mu_{d-1}(y)d\theta 
		\leq \int_{\theta \in {(\mathcal{S}^+)}^{d-1}} \int_{\ell=0}^{4\eps}  C N_S d\ell d\theta
		\leq  4C N_S\eps.
	\end{equation*}
	
\end{proof}

\begin{remark} \label{rem0} If in the proof of Lemma \ref{boundint1} we take $\ell=0$,  then we obtain that the measure of the set of lines belonging to some half-space tangent to $\partial S$ is 0.
\end{remark}

\begin{lemma}\label{propaux}
	Let $S$ be a compact set fulfilling the outside and inside $\alpha$-rolling conditions.
	Let $\X=\{X_1,\dots,X_n\} \subset S$.
	Let $\eps_n\to 0$ be such that $  d_H(\X, S)\leq \eps_n$.
	Let  $r_{\theta,y}=y+\R\theta$ be   any  line fulfilling condition $L(\eps_n)$.
	Then, for $n$ large enough so that $4\varepsilon_n<\alpha$, $n_{\partial S}(\theta,y)=\hat{n}_{\eps_n,\X}(\theta,y)$.
\end{lemma}
\begin{proof} 
	Throughout this proof, we will use the following notation when $r_{\theta,y}\cap \partial S\neq\emptyset$.
	Let $t_1<\ldots<t_{2k}$ be the intersection of $r_{\theta,y}$ with $\partial S$. This set is finite due to Lemma \ref{inter} 
	and is an even number because condition $L(\eps_n)$ is fulfilled. In addition, $[t_{2(i-1)+1},t_{2i}]\subset S$ for all $i=1,\dots,k$ and $(t_{2i},t_{2i+1})\subset S^c$ for all $i=1,\dots,k-1$. 
	
	First, we will prove that
	\begin{equation}\label{hatnsupn}
		\hat{n}_{\eps_n,\X}(\theta,y)\geq n_{\partial S}(\theta,y).
	\end{equation}

	If $r_{\theta,y}\cap \partial S=\emptyset$, then  inequality \eqref{hatnsupn} holds. Assume $r_{\theta,y}\cap \partial S\neq\emptyset$.
	We will now prove that
	\begin{equation}\label{eloi1}
		\text{if }(t_i,t_{i+1})\subset S^c \text{, then: } \exists s \in  (t_i,t_{i+1}) \text{ such that } d(s,S)> 4 \eps_n.
	\end{equation}
	Because $S$ fulfills the inside $\alpha$-rolling condition on $t_i$, there exists a $z_i\in S$ 
	such that $t_i\in \partial \mathcal{B}(z_i,\alpha)$  and $\mathcal{B}(z_i,\alpha)\subset S$.
	Since $\mathcal{B}(z_i,\alpha)\cap(t_i,t_{i+1})=\emptyset$, it follows that $\langle \eta_{t_i}, \theta\rangle \geq 0$ (recall that 
	$\eta_{t_i}=(t_i-z_i)/\alpha$ and $t_{i+1}-t_i=||t_{i+1}-t_i||\theta$).
	Reasoning in the same way but with $t_{i+1}$, $\langle \eta_{t_{i+1}}, \theta\rangle \leq 0$.
	By condition $L(\eps_n)$, we obtain
	\begin{equation}\label{TVI}
		\langle \eta_{t_i}, \theta \rangle > 0 \text{ and } 
		\langle \eta_{t_{i+1}}, \theta\rangle < 0.
	\end{equation}	
	Suppose that  for all $t\in(t_i,t_{i+1})$   we have  $d(t,\partial S)\leq 4\eps_n$.
	Take $n$ large enough so that $4\eps_n<\alpha$.
	Because $S$ fulfills the outside and inside $\alpha$-rolling conditions, by Lemma 2.3 in \cite{bea09}, $\partial S$ has positive reach greater than $\alpha$.
	Then, by Theorem 4.8 in \cite{fed:56}, $\gamma=\{\gamma(t)=\pi_{\partial S}(t),t\in (t_i,t_i+1)\}$, the orthogonal projection onto $\partial S$ of  the  interval $(t_i,t_{i+1})$ is well defined and is a continuous curve in $\partial S$.
	By Theorem 1 in \cite{wal99}, the map
	from $\partial S$ to $\mathbb{R}^d$  $x\mapsto \eta_x$  is Lipschitz.
	Thus, 
	$t\mapsto\langle \eta_{\gamma(t)},\theta \rangle $ is a continuous function of $t$ for all $t\in (t_{i},t_{i+1})$, which, together with  \eqref{TVI}, ensures 
	the existence of an $s\in(t_i,t_{i+1})$ such that $d(s,\gamma(s))\leq 4\eps_n$ 
	and $\theta\in \eta_{\gamma(s)}^{\perp}$, which contradicts the assumption that $r_{\theta,y}$ fulfills condition $L(\eps_n)$.
	This proves  \eqref{eloi1}.
	
	From \eqref{eloi1}, we easily obtain (because $s\in S^c$ and $\X\subset S$) that
	\begin{equation}\label{eloi1bis}
		\text{if }(t_i,t_{i+1})\subset S^c \text{, then: } \exists s \in  (t_i,t_{i+1}) \text{ such that } d(s,\X)> 4 \eps_n.
	\end{equation}
	
	To conclude \eqref{hatnsupn} let us write, for $i=1,\dots,k$, $I'_i=[t_{2(i-1)+1},t_{2i}]$ for the connected components 
	of $S\cap r_{\theta,y}$. Since $d_H(\X,S)<\eps_n$, $S\subset \hat{S}_{\eps_n}(\X)$. Then, for $i=1,\dots,k$, there exists a $j$ such that $I'_i\subset I_j$, $I_j$ being a connected 
	component of $\hat{S}_{\eps_n}\cap r_{\theta,y}$. Note now that \eqref{eloi1bis} ensures that, for all $i\neq i'$, if $I'_i\subset I_j$ and  $I'_{i'}\subset I_{j'}$ then $I_{j'}$ and $I_j$ are not in the same connected component of $\hat{S}_{4\eps_n}(\X)$ thus they are not glued, and then  $\hat{n}_{\eps_n,\X}(\theta,y)\geq n_{\partial S}(\theta,y)$. 
	\\
	Next, we will prove the opposite inequality, 
	\begin{equation}\label{nsuphatn}
		\hat{n}_{\eps_n,\X}(\theta,y)\leq n_{\partial S}(\theta,y).
	\end{equation}
	
	Assume first $r_{\theta,y}\cap \partial S\neq\emptyset$. 
	Consider $t^*\in (t_i,t_{i+1})\subset S^c$ and $t^*\in \hat{S}_{\eps_n}(\X)$.
	Equation  (\ref{nsuphatn}) will be derived from the fact that $(t^*,t_{i+1}]\subset  \hat{S}_{4\eps_n}(\X)  \cap r_{\theta,y}$ or
	$[t_i,t^*)\subset  \hat{S}_{4\eps_n}(\X) \cap r_{\theta,y}$ and thus the connected component of $\hat{S}_{\eps_n}(\X)\cap r_{\theta,y}$ that contained $t^*$ is glued with the one that contains $[t_{i-1},t_i]$ or with the one that  contains $[t_{i},t_{i+1}]$.
	
	Introduce $\psi(t):(t_i,t_{i+1})\rightarrow \mathbb{R}$ defined by $\psi(t)=d(t,\partial S)$.
	Consider the points $t\in (t_i,t_{i+1})$ such that $d(t,\partial S)<\alpha$,  and let $p_t\in \partial S$ be such that $||p_t-t||=d(t,\partial S)$.
	By item (3) in Theorem 4.8 in \cite{fed:56}, $\psi'(t)= \langle \eta_{p_t}, \theta \rangle$.
	
	Let $X_j$ be the closest observation to $t^*$ (recall that because $t^*\in \hat{S}_{\eps_n}(\X)$, we have $||X_j-t^*||\leq \eps_n$). Now, because there exists a point $p^*\in [t^*,X_j]\cap \partial S$, we obtain that $\psi(t^*) \leq \eps_n$ and, because $r_{\theta,y}$ fulfils $L(\eps_n)$, $\langle \eta_{p_{t^*}},\theta \rangle \neq 0$.
	
	Assume that, for instance,  $\langle \eta_{p_{t^*}},\theta \rangle <0$. Then, $\psi(t^*)\leq \eps_n$ and $\psi'(t^*)<0$.
	Suppose that there exists a $t'\in (t^*,t_{i+1})$ such that $\psi(t')\geq \eps _n$ and consider $t''=\inf \{t>t^*, \psi(t')\geq\eps_n\}$. Then for all $t\in (t^*,t'')$, we have $\psi(t)\leq \eps_n<\alpha$, and thus $\psi$ is differentiable on this interval (using again item (3) of Theorem 4.8 in \cite{fed:56}).
	From the fact that $\psi(t'')\geq \psi(t^*)$ and $\psi'(t^*)<0$ we deduce that there exists a $\tilde{t}\in (t^*,t'')$ such that $\psi'(\tilde{t})=0$, which contradicts $L(\eps_n)$ because $\psi(\tilde{t})\leq \eps_n$.
	To summarize, we have shown   that 
	if $\langle \eta_{p_{t^*}},\theta\rangle <0$, then for all $t\in(t^*,t_{i+1})$ we have that $d(t,\partial S)\leq \eps_n$, and thus 
	$(t^*,t_{i+1})\subset  \hat{S}_{2\eps_n}(\X)\subset  \hat{S}_{4\eps_n}(\X)$.
	
	Symmetrically, if $\langle \eta_{p_{t^*}},\theta\rangle >0$, then $(t_i,t^*)\subset  \hat{S}_{2\eps_n}(\X)\subset  \hat{S}_{4\eps_n}(\X)$.
	
	Thus, we now have that if $r_{\theta,y}\cap \partial S\neq \emptyset$, then $\hat{n}_{\eps_n,\X}(\theta,y)\leq n_{\partial S}(\theta,y)$.
	
	Now we are going to prove that for a line  $r_{\theta,y}$ fulfilling condition $L(\eps_n)$ we cannot have  $r_{\theta,y}\cap \partial S=\emptyset$ and $\hat{n}_{\eps_n,\X}(\theta,y)>0$.
	Reasoning by contradiction, upon assuming that  $r_{\theta,y}\cap \partial S=\emptyset$ and $\hat{n}_{\eps_n,\X}(\theta,y)>0$, we have that $0<\min\{||x-y||,x\in r_{\theta,y}, y\in S\}\leq \eps_n$. Now the regularity condition also gives that  if this minimum is realized for $x^*$ and $y^*$, then we have $y^*\in \partial S$ and $\theta \in T_{y^*}\partial S$, which contradicts condition  $L(\eps_n)$. 
\end{proof}

\begin{lemma} \label{lemaux4} Let $S\subset \mathbb{R}^d$ be a compact set
	fulfilling the outside and inside $\alpha$-rolling conditions.
	Let $\X\subset S$ and suppose $\eps_n\to 0$ is a sequence such that $d_H(\X, S)\leq \eps_n$, while $r_{\theta,y}=y + \R \theta$ and $A_1, \ldots, A_k$ are the sets in Definition \ref{hatn},  $A_i=(a_i,b_i)$ for $i=1,\dots,k$. Now suppose  that the sets are indexed in such a way that $a_1<b_1<a_2<\ldots<b_k$.
	Then, for all $i=2,\dots,k$, we have that   $||a_i-b_{i-1}||>3\sqrt{\eps_n\alpha}$ and for all $i=1,\dots,k$, $||b_{i}-a_{i}||>3\sqrt{\eps_n\alpha}$,
	for $n$ large enough so that 	$3\sqrt{\alpha  \eps_n}<\alpha/2$,
	which implies 
	\begin{equation*}
		\hat{n}_{\eps_n,\X}(\theta,y)\leq \frac{\text{diam}(S)}{3\sqrt{\alpha}}\eps_n^{-1/2}.
	\end{equation*}
\end{lemma}

\begin{proof} Assume by contradiction that for some $i$, $||a_i-b_{i-1}||\leq 3\sqrt{\eps_n\alpha}$.
	By construction, $[b_{i-1},a_i]\subset \hat{S}_{\eps_n}(\X)^c \subset S^c$.
	Because $a_i$ and $b_i$ are on $\partial \hat{S}_{\eps_n}(\X)$, we have
	$d(a_i,\X)=d(b_{i-1},\X)=\eps_n$.
	
	The projection $\pi_{S}:[b_{i-1},a_i]\to \partial S$  is uniquely defined because $\partial S$ has reach at least $\alpha$ and $d(t,\partial S)\leq d(t,a_i)+d(a_i,\partial S) \leq ||a_i-b_{i-1}||+d(a_i,\X)$ for all $t\in(b_{i-1},a_i)$, $||a_i-b_{i-1}||\leq 3\sqrt{\eps_n\alpha}<\alpha/2$ and $d(a_i,\partial S)\leq \eps_n\leq \alpha/2$.
	Moreover, $\pi_S$ is a continuous function.
	
	Hence $\max_{x\in [b_{i-1},a_i]}||x-\pi_{ S}(x)||\geq\eps_n-d_H(S,\X)$, and the maximum is attained at some $x_0\in [b_{i-1},a_i]$.
	First, we show that  $||x_0-\pi_{ S}(x_0)||\geq 3\eps_n$. Indeed, suppose by contradiction that
	for all $t\in (b_{i-1},a_i)$, $d(t,\partial S)\leq 3\eps_n$. Then, $d(t,\X)\leq 4\eps_n$, which contradicts the definition of the points $a_i$ and $b_i$.
	The fact that $||x_0-\pi_{ S}(x_0)||\geq 3\eps_n>d(a_i,S)=d(b_{i-1},S)$  guarantees that $x_0\in (b_{i-1},a_i)$ and that $\eta_0$, the outward unit normal vector to 	$\partial S$ at $\pi_{ S}(x_0)$, is normal to $\theta$.
	
	Let  $z_0=\pi_S(x_0)+ \eta_0\alpha$.
	Observe that $d(a_i,S)\leq \eps_n$  and  $d(b_{i-1},S)\leq \eps_n$.
	From the outside $\alpha$-rolling condition at $\pi_S(x_0)$, $||x_0-\pi_S(x_0)||\leq \alpha$ and using the fact that $\eta_0$ is normal to $\theta$, we have (see Figure \ref{figcat})
	\begin{equation*}
		r_{\theta,y}\cap \mathcal{B}(z_0,\alpha-\eps_n)\subset [b_{i-1},a_i],
	\end{equation*}
	
	\noindent which implies, see Figure \ref{figcat}, that 
	$||a_i-b_{i-1}||\geq 2\sqrt{(\alpha-\eps_n)^2-(\alpha-\ell)^2}$, where $\ell=d(x_0,\pi_S(x_0))$.
	Therefore,
	\begin{equation}\label{minol}
		||a_i-b_{i-1}||\geq 2\sqrt{(\ell-\eps_n)(2\alpha-\ell-\eps_n)}.
	\end{equation}
	
	If we bound $\ell\geq 3\eps_n$ and use the fact that $\ell=o(1)$, which follows from $\ell\leq ||b_{i-1}-a_i||+\eps_n\leq 3\sqrt{\eps_n\alpha}+ \eps_n$, then we get, from \eqref{minol}, 
	\begin{equation*}
		||a_i-b_{i-1}||\geq 2\sqrt{2\eps_n(2\alpha-\ell-\eps_n)}= 2\sqrt{4\eps_n\alpha (1+o(1)))}= 4\sqrt{\alpha \eps_n}(1+o(1)),
	\end{equation*}	
	and for $n$ large enough this contradicts  $||a_i-b_{i-1}||\leq  3\sqrt{\alpha \eps_n}$.
	\begin{figure}[ht]%
		\begin{center}
			\includegraphics[scale=0.4]{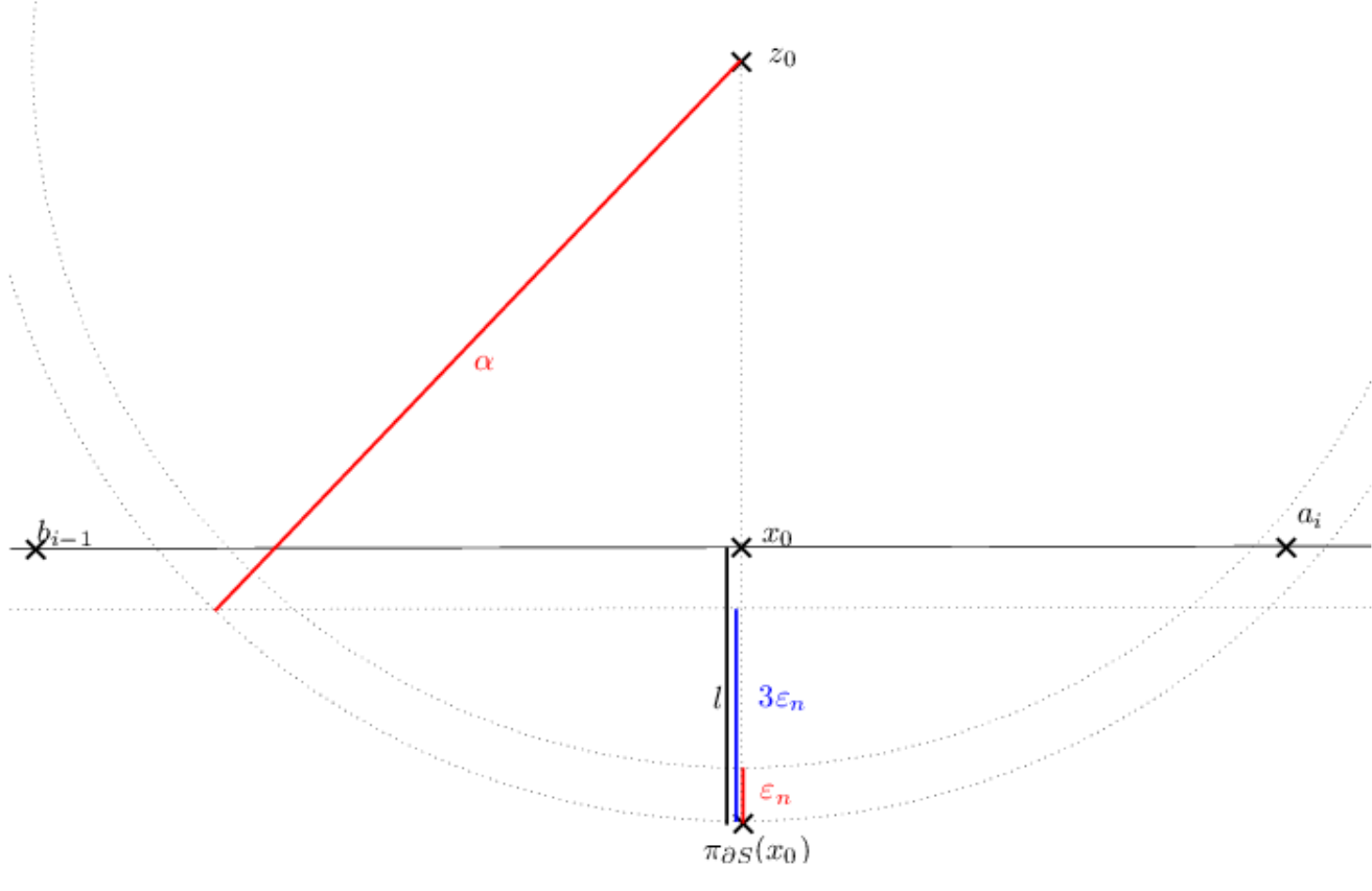} \hspace{3cm} 	
			\caption{$||a_i-b_{i-1}||\geq 2\sqrt{(\alpha-\eps_n)^2-(\alpha-\ell)^2}$, where $\ell=d(x_0,\pi_S(x_0))$.}
			\label{figcat}
		\end{center}
	\end{figure} 
	Then, the number of disjoint intervals $A_i$ is bounded from above by $\text{diam}(S)/(3\sqrt{\eps_n\alpha})$. The proof that for all $i=1,\dots,k$, $||b_{i}-a_{i}||>3\sqrt{\eps_n\alpha}$ follows the same ideas,  we will give a sketch of the proof. 
	Let $b_i>a_i$ (recall that we ordered the points $a_1<b_1<\ldots<a_k<b_k$) be such that $||a_i-b_i||\leq 3\sqrt{\eps_n\alpha}$. Proceeding as before, $\max_{x\in [a_i,b_i]}  ||x-\pi_S(x)||\geq 3\eps_n$ and it 
	is attained at some $x_0\in (a_i,b_i)$. Let $z_0=\pi_S(x_0)-\eta_0\alpha$, with $\eta_0$ being the outward unit normal vector to $\partial S$ at $\pi_S(x_0)$.
	Then $r_{\theta,y}\cap \B(z_0,\alpha)\subset [a_i,b_i]$ since $[a_i,b_i]\notin \B(z_0,\alpha)$ and $\B(z_0,\alpha)\subset S$. From $r_{\theta,y}\cap \B(z_0,\alpha)\subset [a_i,b_i]$ 
	it follows as before that $||a_i-b_i||\geq 4\sqrt{\eps_n \alpha}$, which is a contradiction.
	
	Finally, $\hat{n}_{\eps_n,\X}(\theta,y)\leq \text{diam}(S)/(3\sqrt{\eps_n\alpha})$. 
\end{proof}

\begin{corollary}\label{coraux2}
	Let $S\subset \mathbb{R}^d$ be a compact set
	fulfilling the outside and inside $\alpha$-rolling conditions and with a $(C,\eps_0)$-regular boundary.
	For $n$ large enough so that $3\sqrt{\alpha  \eps_n}<\min(\alpha/2,\eps_0)$, we have
	\begin{equation*}
		\int_{\theta}\int_{y\in \mathcal{A}_{\eps_n}(\theta)} \hat{n}_{\eps_n,\X}(\theta,y)  d\mu_{d-1}(y)d\theta\leq C\frac{\text{diam}(S)}{3\sqrt{\alpha}}\sqrt{\eps_n}.
	\end{equation*}
\end{corollary}

\subsubsection{Proof of Theorem \ref{mainth}}
Without loss of generality, we can assume that $0\in S$.
Recall that for $\theta \in (\mathcal{S}^+)^{d-1}$,  $\mathcal{A}_{\eps_n}(\theta)$ is the set of all $y\in \theta^{\perp}$ such that  $||y||\leq \text{diam}(S)$ and $r_{\theta,y}$ does not fulfill $L(\eps_n)$.
First, from Lemma \ref{propaux}, we have that $	||\partial S|_{d-1} - \hat{I}_{d-1}(\mathbb{X},\eps)|$ is bounded from above by
\begin{equation*}
	\frac{1}{\beta(d)} \int_{\theta \in {(\mathcal{S}^+)}^{d-1}}\int_{y \in \mathcal{A}_{\eps_n}(\theta)} |\hat{n}_{\eps_n,\X}(\theta,y)-n_{\partial S}(\theta,y)| d\mu_{d-1}(y) d\theta,
\end{equation*}
which is bounded from above by
\begin{multline*}
	\frac{1}{\beta(d)} \int_{\theta \in {(\mathcal{S}^+)}^{d-1}}\int_{y \in \mathcal{A}_{\eps_n}(\theta)} \hat{n}_{\eps_n,\X}(\theta,y) d\mu_{d-1}(y) d\theta+\\ \frac{1}{\beta(d)} \int_{\theta \in {(\mathcal{S}^+)}^{d-1}}\int_{y \in \mathcal{A}_{\eps_n}(\theta)} n_{\partial S}(\theta,y)d\mu_{d-1}(y) d\theta.
\end{multline*}

Now, by Corollary \ref{coraux2} and Lemma \ref{boundint1},  we get that
\begin{equation*}
	||\partial S|_{d-1} - \hat{I}_{d-1}(\mathbb{X},\eps)|\leq  \frac{7C \text{diam}(S)}{3\beta(d)\sqrt{\alpha}} \sqrt{\eps_n},
\end{equation*}
for $n$  large enough.

\subsubsection{Proof of Theorem \ref{mainthbis}}

The proof of  Theorem \ref{mainthbis} is basically the same as the previous one. 
Since $N_0\geq N_S$ Lemma \ref{propaux} ensures that, for all $r_{y,\theta}$ not in $\mathcal{A}_{\eps_n}(\theta)$,
$\min(\hat{n}(\theta,y),N_0)=n_{\partial S}(\theta,y)$, for $n$ large enough that $4\eps_n<\alpha$. Thus, we still have, for $n$ large enough, $	||\partial S|_{d-1} - \hat{I}_{d-1}^{N_0}(\partial S)|$ is bounded from above  
\begin{multline*}
	\frac{1}{\beta(d)} \int_{\theta \in {(\mathcal{S}^+)}^{d-1}}\int_{y \in \mathcal{A}_{\eps_n}(\theta)}
	n_{\partial S}(\theta,y) d\mu_{d-1}(y) d\theta+\\
	\frac{1}{\beta(d)} \int_{\theta \in {(\mathcal{S}^+)}^{d-1}}\int_{y \in \mathcal{A}_{\eps_n}(\theta)} N_0 d\mu_{d-1}(y) d\theta.
\end{multline*}

Now, by applying \eqref{boundint1eq2} for the first part and a similar calculation for the second part, we get that
\begin{equation} \label{cota}
	||\partial S|_{d-1} - \hat{I}_{d-1}^{N_0}(\partial S)|\leq  \frac{4C(N_S+N_0)}{\beta(d)}\eps_n,
\end{equation} 
for $n$  large enough.

\subsubsection{Proof of Corollary \ref{maintdh2bis}}
By Corollary 1 in \cite{ch:212}, we know that, with probability one, 
for $T$ large enough, $d_H(\mathbb{X}_T,S)\leq \eps_T\to 0$, where $\eps_T=o((\ln(T)^2/T)^{1/d})$.
Let $\X=\{X_{t_1},\dots,X_{t_n}\}$ be a discretization of $\mathbb{X}_T$ such that $t_i-t_{i-1}=T/n$ and $t_n=T$.
Put $\eps_n=d_H(\X,S)$, then $\eps_n\geq \eps_T$.
It is clear that, for a fixed $T$, $\eps_n$  decreases to $\eps_T$ as $n\to \infty$.
To emphasize the dependence on the set, we will write $\hat{I}_{d-1}^{N_0}(\partial S,\X)$ for the estimator  based on $\X$, and $\hat{I}_{d-1}^{N_0}(\partial S,\mathbb{X}_T)$ for the estimator based on $\mathbb{X}_T$ (both defined using Definition \ref{hatn}).
Then, by \eqref{cota}, to prove Corollary \ref{maintdh2bis} it is enough to prove $\hat{I}_{d-1}^{N_0}(\partial S,\X)\to \hat{I}_{d-1}^{N_0}(\partial S,\mathbb{X}_T)$ as $n\to \infty$, for arbitrary fixed $T$.
Fix $\theta$ and $y$. It is clear that $\hat{n}(\theta,y)(\partial S,\X)\to \hat{n}(\theta,y)(\partial S,\mathbb{X}_T)$ as $n\to \infty$,  and so Corollary \ref{maintdh2bis} follows by the dominated convergence theorem, using the fact that $\min\{\hat{n}(\theta,y),N_0\}\leq N_0$.

\subsection{Proofs for the estimator based on the $\alpha$-hull}

Theorem \ref{mainth2} will be easily obtained from the two following geometric lemmas and Theorem $3$ in \cite{rhull}.

Here, we need to 
introduce some new notation. 
If $f$ is a function, then $\nabla_f(x)$ denotes its gradient and $\mathcal{H}_f$ its Hessian matrix. Given two sets $C,D\subset \mathbb{R}^d$, we write $C\approx D$ if there exists an homeomorphism between $C$ and $D$.
In what follows, $M\subset \mathbb{R}^d$ will be a compact set, and $\mathcal{C}^2$ a $(d-1)$-dimensional manifold (with or without boundary).
Then for all $x$ in $M$, there exists an $r_x>0$ such that either
\begin{enumerate}
	\item[$i)$] for all $r\leq r_x$, $\mathring{\mathcal{B}}(x,r)\cap  M\approx \mathring{\mathcal{B}}_{d-1}(0,1)$, or
	\item[$ii)$] for all $r\leq r_x$, $\mathring{\mathcal{B}}(x,r)\cap M\approx \mathring{\mathcal{B}}_{d-1}(0,1)\cap \{(x_1,\ldots,x_{d-1}):x_1\geq 0\}$.
\end{enumerate}

The set of points satisfying condition $i)$ constitute $\text{int}(M)$, while the set of points satisfying $ii)$ constitute
$\partial M$. We have that $\partial M$ is a $(d-2)$-dimensional manifold without boundary and, as a consequence, $|\partial M|_{d-1}=0$.  

Given a point $x\in M$, $N_x M=\{v\in \mathbb{R}^d:\langle v, u\rangle =0, \forall u\in T_xM\}$ is the $1$-dimensional orthogonal subspace.
If $M$ is a manifold as before, and $\partial M=\emptyset$, we define for any compact set  $E\subset M$ ($E$ is not necessarily a manifold)  its interior 
$\text{int}(E)=\{x\in E: \exists r_x \text{ such that for all }r\leq r_x, \mathring{\mathcal{B}}(x,r)\cap E \approx \mathring{\mathcal{B}}_{d-1}(0,1)\}$.
We have $\text{int}(E)$ is a manifold (without boundary and, when is not empty $\text{int}(E)$ has the same dimension as $M$).

\begin{lemma}\label{rconvhullgeo} 
	Let  $S\subset\mathbb{R}^d$ be a compact set fulfilling the inside and outside $\alpha$-rolling conditions.
	Let $\alpha'<\alpha$ be a positive constant. Let $\X=\{X_1,\ldots,X_n\}\subset S$ be such that :
	
	\begin{enumerate}
		\item[i.] $d_H(\partial C_{\alpha'}(\X),\partial S)\leq \eps_n$ with 
		$\eps_n< \frac{\alpha \alpha'}{2(\alpha+\alpha')}$ (notice that we then have  $\eps_n\leq \alpha'/2$ and $\eps_n\leq \alpha/4$).
		
		\item[ii.] $d_H(\X,S)<\frac{1}{3}\frac{\alpha \alpha'}{\alpha+\alpha'}$ note that  $\frac{1}{3}\frac{\alpha \alpha'}{\alpha+\alpha'}\leq\frac{\alpha'}{3}$
	\end{enumerate}

	Then,
	\begin{enumerate}
		\item there exist $C_1(\X),\ldots,C_K(\X)$ such that:
		\begin{enumerate}
			\item $\bigcup_{i=1}^K C_i(\X)\subset  \partial C_{\alpha'}(\X)$
			\item $|  \partial C_{\alpha'}(\X) \setminus (\bigcup_{i=1}^K C_i(\X))|_{d-1}=0$
			\item  $C_i(\X)$ is a $\mathcal{C}^2$ $(d-1)$-dimensional manifold 
			\item $C_i(\X)\cap C_j(\X)=\emptyset$ when $i\neq j$
		\end{enumerate}
		for all  $x\in \bigcup_{i=1}^K C_i(\X)$,  there exists a $\hat{\eta}_x$, the unit normal
		(to $\partial C_{\alpha'}(\X)$),
		a vector pointing outward (with respect to $C_{\alpha'}(\X)$)  from $x$  that satisfies
		$$\langle \hat{\eta}_x, \eta_{\pi_{\partial S}(x)} \rangle \geq 1-\frac{2(\alpha+\alpha')}{\alpha \alpha'}\eps_n.$$
		\item $\pi_{\partial S}:\partial C_{\alpha'}(\X) \rightarrow \partial S$ the orthogonal projection onto $\partial S$ is one to one.
		\item $\partial C_{\alpha'}(\X)\approx \partial S$
	\end{enumerate}
	
\end{lemma}
\begin{proof}
	Let us prove first that   there are no isolated points in $\partial C_{\alpha'}(\X)$. Indeed, suppose by contradiction that there exists $x$ is an isolated point of 
	$\partial C_{\alpha'}(\X)$; that is, there exists $r>0$ such that $B(x,r)\cap \partial C_{\alpha'}(\X)=\{x\}$.
	By connectedness of $B(x,r)\setminus\{x\}$ we have either $B(x,r)\setminus\{x\}\subset  C_{\alpha'}(\X)^c $
	or  $B(x,r)\setminus\{x\}\subset  C_{\alpha'}(\X)$.  
	The second case   contradicts $x\in \partial C_{\alpha'}(\X)$
	because $C_{\alpha'}(\X)$ is a close set.  
	Thus, we have $B(x,r)\setminus\{x\} \subset C_{\alpha'}(\X)^c$.
	Let us introduce $x^*=\pi_{\partial S}(x)$, then $||x-x^*||\leq \eps_n$. Let us denote $\eta^*=\eta_{x^*}$. 
	Since $\partial C_{\alpha'}(\X)\subset S$, $x\in S$, then by definition of $x^*$, $x+||x-x^*||\eta^*=x^*$. Let us introduce $O=x^*-\alpha \eta^*$. From the inner rolling ball property, $B(O,\alpha) \subset S$.
	Let us define $y=x-\min(r,\eps_n)\eta^*$. From $y\in C_{\alpha'}(\X)^c$ it follows that there exists $O_y$ such that $||O_y-y||<\alpha'$ and
	$B(O_y,\alpha')\cap \X=\emptyset$. From $d_H(\X,S)<\alpha'$ we have $||O_y-O||>\alpha$, and thus $[O,O_y] \cap \partial B(O,\alpha)\neq \emptyset$. 
	Let us define $z=[O,O_y] \cap \partial B(O,\alpha)$, then  $z\in S$ and $B(z,(\alpha'+\alpha-||O_y-O||))\cap \X=\emptyset$. 
	We will prove that $\alpha'+\alpha-||O_y-O||\geq d_H(\X,S)$, which is a contradiction. \\
	
	Because $x\in C_{\alpha'}(\X)$ we have $||O_y-x||\geq \alpha'$.
	Let us write $O_y=y+a\eta^*+bw$ with $||w||=1$ and $w\in (\eta^*)^{\perp}$, by $||O_y-x||\geq \alpha'$ and $||O_y-y||<\alpha'$ it quickly comes that $a^2+b^2\leq (\alpha')^2$ and
	$a\leq \frac{\min(r,\eps_n)}{2}$.
	
	Now 
	\begin{align*}
		||O_y-O||^2&=(\alpha-\min(r,\eps_n)-||x-x^*||+a)^2+b^2\\
		&\hspace{-1cm}\leq(\alpha-\min(r,\eps_n)-||x-x^*||)^2+2a(\alpha-\min(r,\eps_n)-||x-x^*||)+(\alpha')^2\\
		&\hspace{-1cm} \leq (\alpha+\alpha'-\min(r,\eps_n)-||x-x^*||)^2-2(\alpha'-a)(\alpha-\min(r,\eps_n)-||x-x^*||)
	\end{align*}
	Thus,   $||O_y-O||$ is bounded from above by
	\begin{align*}
		&\hspace{-.5cm}\leq(\alpha+\alpha'-\min(r,\eps_n)-||x-x^*||)\sqrt{1-2\frac{(\alpha'-a)(\alpha-\min(r,\eps_n)-||x-x^*||)}{(\alpha+\alpha'-\min(r,\eps_n)-||x-x^*||)^2}}\\
		&\hspace{-.5cm}\leq \alpha+\alpha'-\min(r,\eps_n)-||x-x^*||-\frac{(\alpha'-a)(\alpha-\min(r,\eps_n)-||x-x^*||)}{(\alpha+\alpha'-\min(r,\eps_n)-||x-x^*||)}\\
	\end{align*}
	and
	\begin{align*}
		\alpha+\alpha'-||O_y-O||&\geq (\alpha'-a)\frac{\alpha-\min(r,\eps_n)-||x-x^*||}{\alpha+\alpha'-\min(r,\eps_n)-||x-x^*||}\\
		\alpha+\alpha'-||O_y-O||&\geq \frac{(\alpha'-\eps_n/2)(\alpha-2\eps_n)}{\alpha+\alpha'}\geq \frac{\alpha\alpha'}{\alpha+\alpha'}\frac{3}{8}> \frac{1}{3}\frac{\alpha\alpha'}{\alpha+\alpha'}>d_H(\X,S)\\
	\end{align*}
	
	As announced, this leads to a contradiction.

	From \eqref{rhull} it follows  that, for some $N$, 
	$$\partial C_{\alpha'}(\X)=\bigcup_{i=1}^{N} \Big(\partial B_i \setminus  \bigcup_{j=1}^N B_j\Big).$$
	Here, the $B_i$ are balls of radius $r_i$ larger than $\alpha'$ or half-spaces (by abuse of notation, if $B_i$ is an half-space we will put $r_i=+\infty$).
	
	Our first step consists in proving that:
	\begin{enumerate}
		\item If $x\in \partial C_{\alpha'}(\X)\setminus \X$, then for all $i$ such that $x\in \partial B_i \setminus  \bigcup_{j=1}^N B_j$, we have 
		$r_i=\alpha'$.  
		\item If $x\in \partial C_{\alpha'}(\X)\cap \X$, then there exists an $i$ such that $x\in \partial B_i \setminus  \bigcup_{j=1}^N B_j$ 
		$r_i=\alpha'$.
	\end{enumerate}
	
	We define $S_i=\partial B_i \setminus  \bigcup_{j=1}^N B_j$.
	
	Suppose that $x\in \partial C_{\alpha'}(\X)\setminus \X$. Consider first the case $x\in \partial B_i=\mathcal{S}(O_i,r_i)$
	with $r_i\geq \alpha'$. If $r_i> \alpha'$, then be introducing $\Omega_i=x+(\alpha'/r_i)(O_i-x)=O_i+(r_i-\alpha')(x-O_i)/r_i$ we have that
	$\mathcal{B}(\Omega_i,\alpha')\cap \X \subset (\mathring{\mathcal{B}}(O_i,r_i)\cup\{x\})\cap \X=\emptyset$. Hence, $d(\Omega_i,\X)>\alpha'$,
	and by continuity, there exists a $t>0$ so small that $d(\Omega_i+(t/r_i)(x-O_i),\X)>\alpha'$,
	that is, $\mathcal{B}(\Omega_i+(t/r_i)(x-O_i),\alpha')\subset C_{\alpha'}(\X)^c$ and so $x\in \mathring{C_{\alpha'}(\X)^c}$.
	This is impossible. To conclude this first step, 
	if $x\in \partial C_{\alpha'}(\X)\setminus \X$ with $x\in \partial B_i=\mathcal{S}(O_i,r_i)$, then $r_i=\alpha'$.
	
	Second, consider the case $x\in B_i$ with $B_i=\{z,\langle z-y_i,u_i\rangle >0\}$ where $u_i$ is a unit vector. 
	We can conclude, similarly, on introducing $\Omega_i=x+\alpha'u_i$, that $\mathcal{B}(\Omega_i,\alpha')\cap \X=\emptyset$ 
	and $\mathcal{B}(\Omega_i-tu_i,\alpha')\subset  C_{\alpha'}(\X)^c$ (for some positive but small enough $t$) and so $x\in \mathring{C_{\alpha'}(\X)^c}$.

	If $x\in \partial C_{\alpha'}(\X)\cap \X$, then by   the preliminary result, there exists a sequence 
	$(x_k)$ in $\partial C_{\alpha'}(\X)\setminus \{x\}$ with
	$x_k\rightarrow x$. Because $\X$ is finite, it follows that for $k$ large enough, $x_k\in \partial C_{\alpha'}(\X)\cap \X^c$.
	Because the number of possible $S_i$ is finite, we can extract from $(x_k)$ a sequence $(x'_k)$ such that 
	there exists a $S_i=\partial B_i$ such that for all $k$, $x'_k\in S_i$ making $k\rightarrow +\infty$, and then we have $x\in S_i$.\\

	Our second step consists in proving that if there exists an $x\in \partial B_i \setminus (\bigcup_{j} B_j)$, then 
	\begin{equation}\label{boundps}
		\langle \hat{\eta}_{x,i}, \eta_{\pi_{\partial S}(x)} \rangle \geq 1-\frac{2(\alpha+\alpha')}{\alpha \alpha'}\eps_n,
	\end{equation}
	where $\hat{\eta}_{x,i}=\frac{O_i-x}{\alpha'}$ and $x^*=\pi_{\partial S}(x)$. Observe that from the first step we know that $B_i=\mathring{\mathcal{B}}(O_i,\alpha')$.
	Write $\eta_{x^*}$ for the outward (from $S$) unit normal vector of $\partial S$ at $x^*$ and $O^*=x^*-\alpha \eta_{x^*}$.	
	
	Note first that
	\begin{equation}\label{th2inclu1}
		\mathring{\mathcal{B}}(O_i,\alpha')\subset C_{\alpha'}(\X)^c \text{ and }\mathcal{B}(O^*,\alpha)\subset S.
	\end{equation}
	Introduce $y^*=[O^*,O_i]\cap \partial \mathcal{B}(O_i,\alpha')$ and $y=[O^*,O_i]\cap \partial \mathcal{B}(O^*,\alpha)$ (see Figure
	\ref{thelast}).
	Then, from the  second inclusion in \eqref{th2inclu1}, we get $y\in S$, and from the first inclusion in \eqref{th2inclu1} 
	we get $d(y, C_{\alpha'}(\X))\geq ||y-y^{*}||$. Then,  $ ||y-y^{*}||\leq \eps_n$, which in turn implies 
	\begin{equation}\label{difOO*}
		\alpha+\alpha'-||O_i-O^*||\leq \eps_n.
	\end{equation}
	
	\begin{figure}[h]
		\begin{center}
			\includegraphics[scale=2]{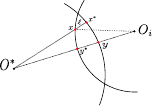}
		\end{center}
		\caption{$x\in \partial C_{\alpha'}(\X)$,   $x^{*}=\pi_{\partial S}(x)$, $O_i=x+\alpha'\hat{\eta}_{x,i}$ and $O^*=x^*-\alpha \eta_{x^*}$.  Observe that $\mathring{B}(O_i,\alpha')\cap B(O^*,\alpha)\neq \emptyset$ by \eqref{inters}}.
		\label{thelast}
	\end{figure}
	From  $x^*=\pi_{\partial S}(x)$  we get that $x^*=x+\ell\eta_{x^*}$ where $\ell=||x-x^*||\leq \eps_n$.
	Then, $O_i=O^*+(\alpha-\ell)\eta_{x^*}+\alpha'\hat{\eta}_{x,i}$ and
	\begin{align}\label{inters}
		\alpha+\alpha'-||O_i-O^*||&=\alpha+\alpha'-\sqrt{(\alpha')^2+(\alpha-\ell)^2+2\alpha'(\alpha-\ell)\langle\hat{\eta}_{x,i},\eta_{x^*}\rangle} \nonumber\\
		&\hspace{-1cm}=\alpha+\alpha'-\sqrt{(\alpha'+\alpha-\ell)^2-2\alpha'(\alpha-\ell)(1-\langle\hat{\eta}_{x,i},\eta_{x^*}\rangle)} \nonumber\\
		&\hspace{-1cm}\geq \ell+\frac{\alpha'(\alpha-\ell)(1-\langle\hat{\eta}_{x,i},\eta_{x^*}\rangle)}{\alpha + \alpha' -\ell}\geq \frac{\alpha'\alpha(1-\langle\hat{\eta}_{x,i},\eta_{x^*}\rangle)}{2(\alpha + \alpha')},
	\end{align}
	where in the first inequality of the last line we bounded $A\sqrt{1-2B/A^2}\leq A(1-B/A^2)=A-B/A$, 
	and in the last inequality $\alpha-\ell\geq \alpha/2$, thus, combined with Equation \eqref{difOO*}, we  can conclude
	the proof of Equation \eqref{boundps}.\\

	As the third step, we will now conclude the proof of assertion $1$.
	Note that if $B_i$ is a ball (and not an half-space), then
	$\partial B_i\cap B^c_j=\partial B_i\cap P_{i,j}$ where $P_{i,j}$ the following closed half space. 
	$$P_{i,j}=\begin{cases} B^c_j \quad \text{ if }B_j\text{ is an open half space}\\
		\{x:||x-O_i||^2-r_i^2\leq ||x-O_j||^2-r_j^2  \}\quad \text{ if }B_j=\mathring{B}(O_j,r_j).\end{cases}
	$$
	Thus,
	$S_i=\partial B_i \bigcap \left( \bigcup_j B_j \right)^c=\partial B_i \bigcap \left( \bigcap_{j\neq i} B_j^c \right)=\partial B_i \cap H_i$, where $H_i$ is a convex polygon.

	Put $C_i(\X)=(\partial B_i \cap \mathring{H}_i)\setminus \X$. We are going to prove that $C_i(\X)$ satisfies conditions $(a)$, $(b)$, $(c)$ and $(d)$ of assertion $1$.
	First note that $(a)$ is obvious by construction.

	Suppose $x\in \partial C_{\alpha'}(\X)\setminus \X$. By the first step, we know that there exists a $B_{i_0}$ which is a ball of radius 
	$\alpha'$ such that $x\in S_{i_0}$ and thus we are in the situation where $x\in \partial B_{i_0} \cap H_{i_0}$
	with $ H_{i_0}$ a convex polygon. If now $x\in \partial C_{\alpha'}(\X)\setminus \X$ but
	$x\notin \cup C_i(\X)$, then we must have $x\in \partial B_{i_0} \cap \partial H_{i_0}$.
	This gives
	$$\partial C_{\alpha'}(\X) \setminus \big(\bigcup_i C_i(\X) \big)\subset \X \bigcup \big(\bigcup_{i,r_i=\alpha'} \partial B_i \cap \partial H_i \big)$$ 
	and thus $|\partial C_{\alpha'}(\X) \setminus \big(\bigcup_i C_i(\X) \big)|_{d-1}=0$, which proves (b).

	We will now prove that if $i\neq j$ and $B_i$ and $B_j$ are two balls, then  $(\partial B_i \cap \mathring{H}_i)\cap (\partial B_j \cap H_j)=\emptyset$. Suppose by contradiction that  
	$(S_i\cap \mathring{H}_i)\cap (S_j\cap H_j) \neq \emptyset$, then $||x-O_i||^2-r_i^2< ||x-O_j||^2-r_j^2$ and 
	$||x-O_i||^2-r_i^2\geq ||x-O_j||^2-r_j^2$, which is a contradiction. Thus, if
	$C_i(\X)$ and $C_j(\X)$ are both non-empty, then we have that $B_i$ and $B_j$ are two balls, and if $i\neq j$, 
	$C_i(\X)\cap C_j(\X)=\emptyset$, which proves $(d)$.

	This also proves that if $x\in C_i(\X)$, then there exists an $r_x>0$ small enough so that 
	$\partial C_{\alpha'}(\X)\cap \mathcal{B}(x,r_x)=\partial B_i \cap \mathcal{B}(x,r_x)$.
	Thus, $\partial C_{\alpha'}(\X)\cap \mathcal{B}(x,r_x)$ is a $\mathcal{C}^2$, $(d-1)$-dimensional manifold. Moreover, the tangent space at $x$ is given by $(x-O_i)^\perp$. Also, the unit normal (to $\partial C_{\alpha'}(\X)$) vector $(O_i-x)/||x-O_i||$ is well defined, and points outwards to $C_{\alpha'}(\X)$. 
	This concludes the proof of $(c)$ and also the proof of 1).
	
	The proof of $2)$ follows the same ideas used to prove Theorem $3$ in \cite{aa:16}.
	We are going to give the main steps of the proof (adapted to our case).
	
	We first prove the subjectivity. For any $x^*\in \partial S$, we introduce $O^*=x^*-\alpha \eta_{x^*}$ and
	$x=x^*-2\eps_n\eta_{x^*}$. From the inside and outside $\alpha$-rolling conditions it follows that $S$ has reach $\alpha>0$, and so $\pi_{\partial S}([x,x^*])=x^*$, where we used that $2\eps_n<\alpha$. To prove that $x\in C_{\alpha'}(\X)$
	we proceed by contradiction. If $x\notin C_{\alpha'}(\X)$, then there exists an $O$ with $||O-x||\leq \alpha'$ and
	$\mathring{\mathcal{B}}(O,\alpha')\subset  C_{\alpha'}(\X)^c$.

	Let $u=(O-x)/||O-x||$,
	$y=[O^*,O]\cap \partial \mathcal{B}(O^*,\alpha)$ and
	$y^*=[O^*,O]\cap \partial \mathcal{B}(O,\alpha')$,
	and therefore  $ ||y-y^{*}||\leq \eps_n$,  which implies 
	\begin{equation}\label{difOO*2}
		\alpha+\alpha'-||O-O^*||\leq \eps_n.
	\end{equation}
	But now
	\begin{align*}
		\alpha+\alpha'-||O-O^*||&= \alpha+\alpha'+||(\alpha-2\eps_n)\eta_{x^*}+u||O-x|| ||\\
		&\hspace{-1.75cm}=\alpha+\alpha'-\sqrt{||O-x||^2+(\alpha-2\eps_n)^2+2||O-x||(\alpha-2\eps_n)\langle u,\eta_{x^*}\rangle}\\
		&\hspace{-1.75cm}=\alpha+\alpha'-\sqrt{(||O-x||+\alpha-2\eps_n)^2-2||O-x||(\alpha-2\eps_n)(1-\langle u,\eta_{x^*}\rangle)}\\
		&\hspace{-1.75cm}=\alpha+\alpha'-(||O-x||+\alpha-2\eps_n)\sqrt{1-\frac{2||O-x||(\alpha-2\eps_n)(1-\langle u,\eta_{x^*}\rangle)}{(||O-x||+\alpha-2\eps_n)^2}}\\
		&\hspace{-1.75cm}\geq\alpha+\alpha'-(||O-x||+\alpha-2\eps_n)\left(1-\frac{||O-x||(\alpha-2\eps_n)(1-\langle u,\eta_{x^*}\rangle)}{(||O-x||+\alpha-2\eps_n)^2}\right)
		\\			
		&\hspace{-1.75cm}\geq 2\eps_n+\alpha'-||O-x||+\frac{||O-x||(\alpha-2\eps_n)(1-\langle u,\eta_{x^*}\rangle)}{(||O-x||+\alpha-2\eps_n)} \geq 2\eps_n,
	\end{align*}

	\noindent where the last inequality follows from $||O-x||\leq \alpha'$ and $2\eps_n<\alpha$. This contradicts Equation \eqref{difOO*2}. Thus,  $x\in C_{\alpha'}(\X)$. From the outside and inside $\alpha$ rolling condition  (which implies $\alpha$-convexity, see \cite{cue12}) it follows that,
	$$C_{\alpha'}(\X)\subset C_{\alpha}(\X)\subset C_{\alpha}(S)=S.$$
	Then,  if $x^*\in \partial S$ then $x^*\in \partial C_{\alpha'}(\X)$ or 
	$x^*\in C_{\alpha'}(\X)^c$. In both cases, there exists  a $z\in (x,x^*)$ and $z\in \partial C_{\alpha'}(\X)$,  such that $\pi_{\partial S}(z)=x^*$.
	
	We now prove the injectivity. Suppose by contradiction that there are $x_1,x_2$ $\in \partial C_{\alpha'}(\X)$ such that $\pi_{\partial S}(x_1)=\pi_{\partial_S}(x_2)=y$. Write $\ell_i=d(x_i,\partial S)$, for $i=1,2$. Because $ C_{\alpha'}(\X)\subset S$, we have 
	$x_i+\ell_i \eta_{y}=y$ and thus
	$x_1=x_2+(\ell_2-\ell_1)\eta_y$ and $|\ell_2-\ell_1|\leq\eps_n$ (because for $i\in\{1,2\}$, 
	$\ell_i\geq 0$ and $\ell_i\leq \eps_n$, due to $d_H(\partial S,\partial C_{\alpha'}(\X))\leq \eps_n$).
	Suppose that $\ell_2\geq \ell_1$.
	From the first step together with Equation \eqref{boundps}, we know that there exists an $O_i$ such
	that $\mathring{\mathcal{B}}(O_i,\alpha')\subset C_{\alpha'}(\X)^c$, $||x_2-O_i||=\alpha'$ and  $\langle u, \eta_{y} \rangle \geq 1-2(\alpha+\alpha')/(\alpha \alpha')\eps_n$ with $u=(O_i-x_2)/\alpha'$. Then 
	\begin{align*}
		||x_1-O_i||^2=&\ (\ell_2-\ell_1)^2+\alpha'^2-2\alpha' (\ell_2-\ell_1) \langle u,\eta_y\rangle \\
		\leq \ &(\ell_2-\ell_1)^2+\alpha'^2-2\alpha' (\ell_2-\ell_1) + \frac{4(\alpha+\alpha')}{\alpha }(\ell_2-\ell_1) \eps_n\\
		\leq \  & \alpha'^2-(\ell_2-\ell_1)\left(2\alpha'-\frac{4(\alpha+\alpha')}{\alpha }\eps_n -(\ell_2-\ell_1) \right)\\
		\leq \ & \alpha'^2-(\ell_2-\ell_1)\left(2\alpha'-\left(\frac{4(\alpha+\alpha')}{\alpha }+1\right)\eps_n \right).
	\end{align*}
	The condition $\eps_n\leq \frac{\alpha \alpha'}{4(\alpha+\alpha')}$ guarantees  $2\alpha'-(4(\alpha+\alpha')/\alpha+1)\eps_n > 0$ thus, if $\ell_2>\ell_1$, then $x_1\in \mathring{\mathcal{B}}(O_i,\alpha')$, 
	which is impossible (recall that $\mathring{B}(O_i,\alpha')\subset C_{\alpha'}^c$ and 
	that $x_1\in C_{\alpha'}(\X) $).
	Thus, by contradiction, $\ell_1=\ell_2$ and $x_1=x_2$, which concludes the proof of injectivity.\\
	
	Finally, we  prove $3.$ Since $reach(\partial S)\geq \alpha$ and $d_H(\partial C_{\alpha'}(\X),\partial S)\leq \eps_n <\alpha$,  $\pi_{\partial S}$, restricted to $\partial C_{\alpha'}(\X)$, is continuous (see \cite{fed:56}).
	The continuity of  $\pi_{\partial S}^{-1}: \partial S \rightarrow \partial C_{\alpha'}(\X)$ follows from the same ideas used to prove the injectivity of $\pi_{\partial S}$: we provide a sketch of the proof.
	It follows from $reach(\partial S)\geq \alpha$ that $\pi_{\partial S}^{-1}(x)=x-\ell(x)\eta_x$ with $\ell(x)\geq 0$.
	In addition, $x\mapsto \eta_x$ is a continuous function (see Theorem 1 in \cite{wal99}).
	It remains to be proved that $\ell$ is a continuous function.
	If this is not the case, then we can find  sequences $(y_n)\subset \partial S$ and $(y'_n)\subset \partial S$, both converging to some $y\in \partial S$), such that $\ell(y_n)\rightarrow \ell_1$ and $\ell(y'_n)\rightarrow \ell_2$.
	We can conclude exactly as in the proof of injectivity that we can take $x_{1,n}=y_n-\ell(y_n)\eta_{y_n}$ and $x_{2,n}=y'_n-\ell(y'_n)\eta_{y'_n}$ making $n\rightarrow + \infty$.
	We thus have $\partial S \approx \partial C_{\alpha'}(\X)$, which proves assertion 3, and thus concludes the proof of the lemma.
	
\end{proof} 

\begin{lemma}\label{NewGeometrie}
	Suppose that $M$ is a $\mathcal{C}^2$, bounded $(d-1)$-dimensional manifold with positive reach $\alpha$. Let 
	$\pi_M$ denote the projection onto $M$ and $\hat{M}$ be a $\mathcal{C}^2$, $(d-1)$-dimensional manifold such that
	\begin{enumerate}
		\item $\pi_M$ is one to one from $\hat{M}$ to $M$,
		\item for all $x\in \hat{M}$ we have $||x-\pi_{M}(x)||\leq \eps_1$ and $\langle \hat{\eta}_x, \eta_{\pi_M(x)}\rangle \geq 1-\eps_2$.
	\end{enumerate}
	Then, if $\eps_1(d-1)\alpha \leq 1$ and $\eps_2\leq 1/8$, we have
	\begin{equation}\label{newencadre}
		\left(1- 3\eps_1 \alpha- 32\eps_2\right)^{\frac{d-1}{2}}\leq \frac{|\hat{M}|_{d-1}}{|M|_{d-1}}\leq \left(1+ 3\eps_1 \alpha+ 32\eps_2 \right)^{\frac{d-1}{2}}.
	\end{equation}
\end{lemma}
\begin{proof}
	Let $p\in M$ and denote by $(e_1,\ldots,e_{d-1})$ an orthonormal basis of $T_pM$ and
	complete it with $e_d$ a unit vector of $N_pM$.
	A neighbourhood of $p$ in $M$ can be parametrized by $\varphi(x)=x+f(x)e_d=\sum_1^{d-1} x_ie_i +f(x_1,\ldots, x_{d-1})e_d$
	where $x=\sum_1^{d-1}x_i e_i$  belongs to a neighborhood of $p$  and $\nabla_f(p)=0$,  see for instance Proposition 3, point 1, in \cite{aach20}.

	Consider now the surface element (of $M$) $ds(p)=dx_1\ldots dx_{d-1}$. Its image by $\pi_{M}^{-1}$ on the surface element (of $\hat{M}$) is given by
	$$d\hat{s}(p)=\sqrt{\det(J_{\pi_{M}^{-1}}(p)'J_{\pi_{M}^{-1}}(p))}dx_1\ldots dx_{d-1}.$$
	The rest of the  the proof consist in giving bounds for 
	$\det(J_{\pi_{M}^{-1}}(p)'J_{\pi_{M}^{-1}}(p))$.  
	We have that $\pi_{M}^{-1}(\varphi(x))=x+\ell(x) n(x)$ where $n(x)=(-\partial f/\partial x_1,\ldots, -\partial f/\partial x_{d-1},1)\in N_xM$, which gives
	that 
	$$J_{\pi_{M}^{-1}}(p)=
	\begin{pmatrix}
		I_{d-1}-\ell(p)\mathcal{H}_{f}(p)\\
		\nabla_{\ell}(p)
	\end{pmatrix}.
	$$	
	The reach condition gives that $||\mathcal{H}_{f}(p)||_{op}\leq \alpha$ (see Proposition 6.1 in \cite{smale}) and 
	$\ell(p)=||\pi_{M}^{-1}(p)-p||\leq \eps_1$ so that we just have to bound $||\nabla_{\ell}(p)||$.
	Note that, for $j=1,\ldots,d-1$, we have
	$$t_j=e_j+\frac{\partial \ell}{\partial x_j}(p)e_d-\ell(p)\left(\sum_{1}^{d-1} \frac{\partial ^2 f}{\partial x_i \partial x_j}
	e_i\right)\in T_{\pi_M^{-1}(p)}\hat{M}.$$
	
	Note that $\eta_{p}=\pm e_d$ and introduce $\hat{\eta}_{\pi_M^{-1}(p)}$. Since $t_1,\ldots,t_{d-1}, \hat{\eta}_{\pi_M^{-1}(p)}$ is an 
	orthogonal basis of $\mathbb{R}^d$, we have that
	$$e_d=\sum_{i=1}^{d-1}\langle e_d, \frac{t_i}{||t_i||}\rangle \frac{t_i}{||t_i||}+\langle e_d,\hat{\eta}_{\pi_M^{-1}(p)} \rangle \hat{\eta}_{\pi_M^{-1}(p)},$$
	which implies
	$$1=\sum_{i=1}^{d-1}\langle \eta_{p}, \frac{t_i}{||t_i||}\rangle ^2+\langle \eta_{p},\hat{\eta}_{\pi_M^{-1}(p)} \rangle ^2. $$

	Thus, by condition $2$, we have $|\langle t_j, e_d \rangle |=|\langle t_j, \eta_p\rangle |\leq \sqrt{2\eps_2} ||t_j||$, 
	which implies
	$$\left|\frac{\partial \ell}{\partial x_j}(p)\right|\leq \sqrt{2 \eps_2} ||t_j|| \leq \sqrt{2 \eps_2} \left(1+\left|\frac{\partial \ell}{\partial x_j}(p)\right|+\eps_1(d-1)\alpha \right).$$
	
	From this, we  get
	$$\left|\frac{\partial \ell}{\partial x_j}(p)\right|\leq \frac{\sqrt{2 \eps_2}(1+\eps_1(d-1)\alpha)}{1-\sqrt{2 \eps_2}}.$$
	So, $J_{\pi_{M}^{-1}}(p)'J_{\pi_{M}^{-1}}(p)=I_{d-1}+E$ with $E$ a symmetric matrix with 
	$$||E||_{op}\leq 2\eps_1 \alpha +\eps_1^2 \alpha^2+ \left(\frac{\sqrt{2 \eps_2}(1+\eps_1(d-1)\alpha)}{1-\sqrt{2 \eps_2}} \right)^2,$$
	thus we finally obtain the inequality 
	$$\left(1- 3\eps_1 \alpha- 32\eps_2\right)^{d-1}\leq \det\left(J_{\pi_{M}^{-1}}(p)'J_{\pi_{M}^{-1}}(p)\right)\leq \left(1+ 3\eps_1 \alpha+ 32\eps_2 \right)^{d-1},$$
	which concludes the proof. \end{proof}

\subsubsection{Proof of Theorem \ref{mainth2}}
Theorem \ref{mainth2} follows now from the previous lemmas. 

Let $\eps_n=d_H(\partial C_{\alpha'}(\X),\partial S)$ and $S_i=\pi_{\partial S}(C_i(\X))$, where the $C_i(\X)$ are the sets introduced in Lemma \ref{rconvhullgeo}, we have
that, for all $i$: $d_H(S_i,C_i(\X))\leq \eps_n$. 
Due to Lemma \ref{rconvhullgeo},
we also have 
\begin{enumerate}
	\item $|\partial S|_{d-1}=\sum_i|S_i|_{d-1}$ and  $|\partial C_{\alpha'}(\X)|_{d-1}=\sum_i|C_i(\X)|_{d-1}$. 
	\item for every $i$ and all $x\in C_i(\X)$, $\langle \hat{\eta}_x, \eta_{\pi_{\partial S}(x)} \rangle\geq 1-\frac{2(\alpha+\alpha')}{\alpha \alpha'} \eps_n$.
\end{enumerate}

Thus, by Lemma \ref{NewGeometrie} we also have, for all $i$:

$$ \left(1-3\alpha\eps_n-\frac{64(\alpha+\alpha')}{\alpha \alpha'} \eps_n\right)^{\frac{d-1}{2}} \leq \frac{|C_i(\X)|_{d-1}}{|S_i|_{d-1}}\leq \left(1+3\alpha\eps_n+\frac{64(\alpha+\alpha')}{\alpha \alpha'} \eps_n\right)^{\frac{d-1}{2}}.$$

We then introduce $A=3\alpha+\frac{64(\alpha+\alpha')}{\alpha \alpha'}$, summing all the terms in the inequalities 

$$ \left(1-A \eps_n\right)^{\frac{d-1}{2}}|S_i|_{d-1} \leq |C_i(\X)|_{d-1}\leq \left(1+A\eps_n\right)^{\frac{d-1}{2}}|S_i|_{d-1}, $$

gives
$$ \left(1-A \eps_n\right)^{\frac{d-1}{2}}|\partial S|_{d-1} \leq |\partial C_{\alpha'}(\X)|_{d-1}\leq \left(1+A\eps_n\right)^{\frac{d-1}{2}}|\partial S|_{d-1}. $$
which concludes the proof.

\subsubsection{Proof of Corollary \ref{mainth2iid}}

We only need to check that the conditions of Theorem \ref{mainth2} are fulfilled, with probability one, for $n$ large enough. 
In \cite{crc:04} it is proved that 
$d_H(\X,S)\leq O((\ln n/n)^{1/d}$ e.a.s. so, with probability one for $n$ large enough it is upper bounded by $\frac{1}{3}\frac{\alpha\alpha'}{\alpha+\alpha'}$. 
In  \cite{rhull} it is proven that, with probability one for $n$ large enough, $d_H(\partial C_{\alpha'}(\X),\partial S)\leq \eps_n \leq c(\ln n/n)^{2/(d+1)}$ for
some given explicit constant $c$. 
Since $C_{\alpha'}(\X)^c$ is  a finite union of balls and affine half-spaces---that is,
$C_{\alpha'}(\X)^c=\bigcup_{i=1}^{N_1} E_i$ with  $E_i=\mathring{\mathcal{B}}(O_i,r_i)$ or 
$E_i=\{z \in \mathbb{R}^d, \langle u_j,z \rangle > a_i \}$---,
it follows that
$$\partial C_{\alpha'}(\X)=\bigcup_i \Big(\partial E_i \bigcap \Big( \bigcup_{j\neq i} E_j \Big)^c\Big).$$

Now define the $F_j$ as the connected components of the sets $\partial E_i \bigcap ( \bigcup_{j\neq i} E_j )^c$. Then, the $F_j$ are 
closed manifolds of dimension $d_j\leq (d-1)$, and are compact since $F_j \subset C_{\alpha'}(\X)$, which is compact.
Finally, because $\partial C_{\alpha'}(\X)$ is a $(d-1)$-dimensional manifold, we must have 
$\partial C_{\alpha'}(\X)=\cup_{j,d_j=d-1} F_j$ (i.e., the lower dimensional $F_k$ are included in $\cup_{j,d_j=d-1} F_j$). 
This concludes the proof of Corollary \ref{mainth2iid}.

\section*{Acknowledgements}

We would like to thank a referee and the AE for careful proofreading, and their helpful and positive suggestions    in a previous version of this manuscript. This research has been partially supported by grant FCE-1-2019-1-156054 (ANII-Uruguay).

\end{document}